\documentclass{article}
\usepackage{amssymb, mathtools,amsthm,microtype}
\usepackage{marginnote}
\usepackage[marginpar=5cm,top=2cm,bottom=3cm, a4paper]{geometry}
\usepackage{xcolor}
\colorlet{darkblue}{blue!50!black}
\usepackage[colorlinks,citecolor=darkblue,linkcolor=darkblue,urlcolor=darkblue]{hyperref}
\usepackage[utf8]{inputenc}
\usepackage[]{natbib}
\usepackage[T1]{fontenc}
\usepackage{subcaption} 
\usepackage{comment}
\usepackage{authblk}
\usepackage{caption} 

\renewcommand{\leq}{\leqslant}
\renewcommand{\geq}{\geqslant}
\renewcommand{\epsilon}{\varepsilon}

\newcommand{\Ac}{\mathcal A}
\newcommand{\Bc}{\mathcal B}
\newcommand{\C}{ C}
\newcommand{\cou}{\Psi}
\newcommand{\cdecou}{C_{n}}

\newcommand{\ccou}{C_{n,T}}
\renewcommand{\d}{\mathrm d}
\newcommand{\dlp}{d_{\mathrm{LP}}}
\newcommand{\dtv}{d_{\mathrm{TV}}}
\newcommand{\decou}{\Phi}

\newcommand{\IDV}{I_{\mathrm{DV}}}
\newcommand{\IRL}{I_{\mathrm{RL}}}

\newcommand{\Jc}{\mathcal J}
\newcommand{\leb}{\mathrm {Leb} }
\newcommand{\LP}{\mathrm{LP}}

\newcommand{\N}{\mathbb N}
\renewcommand{\P}{\mathbb P}
\newcommand{\Pc}{\mathcal P}
\newcommand{\R}{\mathbb R}
\newcommand{\Rxinit}{\R^d \cup\{\xinit\}}
\newcommand{\Sc}{\mathcal S}
\newcommand{\sli}{F}
\newcommand{\sti}{G}
\DeclareMathOperator{\supp}{\mathrm{supp}}
\newcommand{\TV}{{\mathrm {TV}}}
\newcommand{\W}{\widetilde{\mathcal{W}}}
\newcommand{\Wc}{\mathcal W}
\newcommand{\xinit}{x_{\mathrm{init}}}

\newtheorem{hyp}{Assumption}

\newtheorem{hypU}{Assumption}

\newtheorem{theo}{Theorem}[section]
\newtheorem{prop}[theo]{Proposition}
\newtheorem{defi}[theo]{Definition}
\newtheorem{lem}[theo]{Lemma}
\newtheorem{coro}[theo]{Corollary}

\newtheoremstyle{exstyle}{5pt}{10pt}{\rmfamily}{}{\bf}{.}{ }{}
\theoremstyle{exstyle}
\newtheorem{example}[theo]{Example}

\numberwithin{equation}{section}
\mathtoolsset{showonlyrefs}

\title{Large deviations for non-irreducible Markov chains on Euclidean spaces}
\author[1]{L\'eo Daures\thanks{Corresponding author: daures@lpsm.paris}}
\date{}
\affil[1]{\small Université Paris Cité and Sorbonne Université, CNRS, Laboratoire de Probabilités, Statistique et Modélisation, Paris, France}
\begin{document}
	\maketitle
	\begin{abstract}
		We establish the weak large deviation principle for empirical measures of Markov chains on $\mathbb{R}^d$ under mild assumptions.
		In particular, no irreducibility is assumed and the initial measure may be arbitrary. 
		The proof is entirely self-contained and relies on subadditivity.
		In the absence of irreducibility, examples show that the rate function is not convex in general.
	\end{abstract}
	\vspace{1cm}	
	
\tableofcontents
\section{Introduction}
Let $(X_n)_{n\in \N}$ be a Markov chain on $\R^d$, $d\geq 1$, associated with a stochastic kernel $p$ and an initial probability measure $\beta$. The \emph{empirical measure} of $(X_n)$ up to time $n$ is the random probability measure 
\begin{equation*}
	L_n:=\frac 1n\sum_{i=1}^n\delta_{X_i},
	\qquad n\geq 1,
\end{equation*}
where $\delta_x$ denotes the Dirac measure at $x$.
Our goal is to study the large deviations of the sequence $(L_n)$.
A large deviation principle (LDP) is the statement of the exponential decay of some probabilities involving $L_n$.

\begin{defi}
	Let $\Pc(\R^d)$ denote the set of probability measures equipped with its weak topology and the corresponding Borel $\sigma$-algebra.
	We say that $(L_n)$ satisfies the full LDP (resp. weak LDP) in the weak topology if there exists a lower semicontinuous function $I$, called the rate function, such that all open subsets $O$ and all closed subsets $F$ (resp. compact subsets $F$) of $\Pc(\R^d)$ satisfy
	\begin{align}
		\liminf_{n\to \infty}\frac 1n\log \mathbb P(L_n\in O)&\geq -\inf_{\mu\in O}I(\mu),
		\\\limsup_{n\to \infty}\frac 1n\log \mathbb P(L_n\in F)&\leq -\inf_{\mu\in F}I(\mu).
	\end{align}
\end{defi}
The question of the large deviations of $(L_n)$ was first explored in a series of seminal articles by Donsker and Varadhan in the 1970s \cite{DV1,DV2,DV3}.
Since then, a vast literature has been dedicated to establishing this LDP under various assumptions, and computing the rate function;
see the notes of \cite[Chap. 6]{DZ} or  \cite{deacosta2022} for a historical review.
In the literature, most works study the large deviations under two assumptions: one of exponential tightness and one of irreducibility of the Markov chain.

Exponential tightness is the assumption that for all $M>0$, there exists a compact subset $K$ of $\Pc(\R^d)$ such that 
\begin{equation*}
	\limsup_{n\to \infty}\frac 1n\log \P(L_n\notin K)\leq -M.
\end{equation*}
Exponential tightness is often not assumed directly, but rather as the consequence of a broader assumption: for instance, compactness of the state space in \cite{DV1}, Hypothesis (H*) in \cite{DV3} or Assumption~(U), whose definition is recalled below, in \cite{stroock1984,ellis1988,DZ}.
Usually, the primary purpose of exponential tightness is to deduce the full LDP from the weak one: one would prove the weak LDP (without relying on exponential tightness) and then enhance it to the full LDP by the use of an exponential tightness assumption.\footnote
{In fact, the weak LDP has long been mostly used as a step in deriving the full LDP. This is for instance the historical method of \cite{DV3}. {While the term ``weak LDP'' had not been coined yet, the upper bound is indeed proved only for compact sets before exponential tightness is used.}} However, in some cases (which are not exponentially tight \emph{e.g.} the simple random walk on $\mathbb Z$), the full LDP fails while the weak one holds \cite{famex}; in these cases, the weak LDP is the best achievable result and provides valuable information on the decay of probabilities. In the present article, we focus on the weak LDP. As a consequence, we never address the questions of exponential tightness.

The second assumption is that of irreducibility. This assumption can take many different forms.
One classical irreducibility assumption is that of $\psi$-irreducibility \cite{neynummelin1987,deacosta1988,deacosta2022}. Assumption~(U) from \cite{stroock1984,ellis1988,DZ} also works as an irreducibility assumption.
However, not all Markov chains are irreducible, whatever definition of irreducibility is adopted. The simplest examples are of course discrete \cite{dinwoodie1993,deacosta2022,daures2025}, but non-irreducibility is also common for continuous Markov chains.
In particular, non-irreducible Markov chains arise from applied competitive models, where several species (or players) may eventually go extinct (or bankrupt); we present a stochastic version of the Lotka-Volterra competitive system in Example~\ref{ex: lotka volterra}. 
Non-irreducibility can also arise in a perturbed dynamical system; see Example \ref{ex: one dimensional}.
Very few theoretical results exist in the literature to handle large deviations for such non-irreducible Markov chains. 
Recently, the assumption of irreducibility has been removed from statements of the weak LDP in the discrete case \cite{wu2005,rassoul,daures2025}, but the question remains open in the continuous case. 

In the present article, we prove that $(L_n)$ satisfies the weak LDP without any irreducibility condition or exponential tightness assumption. 
In essence, we prove that, if the initial measure $\beta$ and all measures $p(x,\cdot)$ have densities that are both lower semicontinuous and bounded on $\R^d$, uniformly in $x\in \R^d$, and if a few additional mild assumptions are satisfied, then $(L_n)$ satisfies the weak LDP. We identify a set $\Ac\subseteq \Pc(\R^d)$ outside of which the rate function $I$ is infinite, and provide convexity properties of $I$ over $\Ac$. The rigorous statements are Theorems~\ref{theo: main result} and~\ref{theo: main result (simplified)}.

The proof of Theorems~\ref{theo: main result} and~\ref{theo: main result (simplified)} is based on the \emph{subadditive method}. 
This technique consists in applying a version of the subadditive lemma (see \cite[Lemma 6.1.11]{DZ}) in order to prove the existence of the \emph{Ruelle-Lanford function}, which, in turn, implies the desired weak LDP.
The subadditive method originates from early work \cite{azencott,bahadurzabell1979} on i.i.d.~sequences and was later adapted for Markov chains \cite{stroock1984,ellis1988,DZ}, under an irreducibility assumption: this is again the uniformity Assumption~\ref{assumption U}, which we recall for later discussion but will never use in the present article.
\begin{hypU}
	\label{assumption U}
	There exist a constant $c$ and integers $k\leq \tau$ such that
	\begin{equation}
		\label{eq: assumption U}
		\frac c\tau\sum_{i=1}^\tau p^i(x_1,\cdot)\geq p^k(x_2,\cdot), \qquad x_1,x_2\in \R^d.
	\end{equation}
\end{hypU}
\noindent 
Assumption~\ref{assumption U} imposes conditions on the irreducibility of the Markov chain that are so strong that they guarantee not only the LDP\footnote{The LDP obtained under~\ref{assumption U} is the full LDP, since~\ref{assumption U} also implies exponential tightness.} of $(L_n)$ but also its uniformity with respect to the initial measure.
Such strong conditions are not satisfied in general; see Appendix~\ref{section: examples}.
However, subadditive arguments can still be used without~\ref{assumption U}, at the cost of a weaker result.
Recently, \cite{daures2025} proposed adaptations to the subadditive method to prove that, if the state space is discrete, then $(L_n)$ satisfies the weak LDP, without any assumptions other than discreteness.
In the present article, we follow a similar route to prove the weak LDP in the case of continuous state space.
The arguments presented here would continue to work if the state space were discrete and the Lebesgue measure were replaced by the counting measure, effectively providing an alternative formulation of the subadditive proof of the weak LDP of \cite{daures2025}. We chose not to include this generalization and to fix the reference space to $(\R^d,\leb)$ instead.

The subadditive method is often observed to produce a convex rate function. Nevertheless, in the present article, we find the rate function not to be generally convex (see Theorem~\ref{theo: main result} and Example~\ref{ex: lotka volterra}). This possible lack of convexity is related to the lack of irreducibility of the Markov chain. 
To our knowledge, the present article and \cite{daures2025} are the only two successful attempts at using the subadditive method to derive an LDP that has a non-convex rate function.

The article is structured as follows. In the rest of the present section, we introduce our notations and assumptions. We also define the admissible measures and state the main theorem of the paper in Section~\ref{section: mai results}. 
In Section~\ref{section: Ruelle Lanford}, we present the Ruelle-Lanford function. We show that the main theorem follows immediately from the existence of the Ruelle-Lanford function and its properties. The existence of the Ruelle-Lanford function at non-admissible measures is proved in Section~\ref{section: non admissible}. The case of admissible measures requires more involved work. In Section~\ref{section: slicing and stitching}, we present the tools for this work: the slicing, stitching and coupling maps. They are manipulations of trajectories that separate and reassemble portions of trajectories. In Section~\ref{section: coupling and decoupling}, we use the tools from Section~\ref{section: slicing and stitching} to prove the existence of the Ruelle-Lanford function for admissible measures. 
As in other uses of the subadditive method, a property of convexity of the rate function\footnote{We do not prove that the rate function is convex~---~it may not be, in view of the discussion above and Example~\ref{ex: lotka volterra} for instance. However, it still satisfies a convexity inequality between some properly chosen points.}
 should come with the proof of the existence of the Ruelle-Lanford function.
The proof of the weak LDP is complete at the end of Section~\ref{section: coupling and decoupling}.
Finally, in Section~\ref{section: decoupling}, we use the slicing and stitching maps from Section~\ref{section: slicing and stitching} again to prove an additional property of the rate function, which completes the main theorem.

\subsection{Notation}
\label{section: notations}
\subsubsection{Markov chain and words}
We will need some notation in the following.
To begin with, we will consider an additional state $\xinit\notin \R^d$ and set $p(\xinit, \cdot)=\beta$, so that the Markov chain of kernel $p$ that starts at $\xinit$ at time $n=0$ behaves exactly like the Markov chain of kernel $p$ that starts following $\beta$ at time $n=1$. We will not distinguish between the two; this alternative way to see the initial measure will be convenient for later notation and definitions.

We will also denote bits of trajectories of $(X_n)$ by \emph{words}.
In the following, $\Wc_n$ denotes the set of words of length $n$, that is $(\R^d)^n$. The length of a word is denoted by $|u|$. The empty word is denoted by $e$. The set of all (finite) words is $\Wc:= \bigcup_{n\geq 0}\Wc_n$, and the set of all words of length at most $n$ is $\Wc_{\leq n}:= \bigcup_{0\leq k\leq n}\Wc_k$. The $i$-th letter of a word $u$ is denoted by $u_i$. We denote the last letter of $u$ by $u_{-1}=u_{|u|}$.

We equip $\Wc$ (resp. $\Wc_n$) with its Borel $\sigma$-algebra $\Bc(\Wc)$ (resp. $\Bc(\Wc_n)$).
For all $n\in \N$ and given any $x\in \Rxinit$, we define the probability measure $p(x,\cdot)$ over $(\Wc_n,\Bc(\Wc_n))$ by
\begin{align}
	\label{eq: p on words}
	p(x,\d u)&:=p(x,\d u_1)\prod_{i=1}^{n-1}p(u_i,\d u _{i+1}).
\end{align}
For notational convenience, we also set $p(x,\cdot )=\delta_e$ on $\Wc_0=\{e\}$.\footnote{This convention will be very useful in the proof of Proposition~\ref{prop: proba ineq coupling map}.} 
In the special case $x=\xinit$, we write 
\begin{equation*}
	\P(\d u):=
	p(\xinit,\d u).
\end{equation*}
By extension, $\P$ is the law of the Markov chain $(X_n)$.

It will be useful to consider the set of words of $\Wc$ that contribute to the support of $\mathbb P$: the definition of such a set $\W$ is given in Section~\ref{section: kernel and densities}.
\subsubsection{Kernels and densities}
\label{section: kernel and densities}
In the following, being able to manipulate densities instead of the measures $p(x,\cdot)$, $p^k(x,\cdot)$ and $\tilde p(x,\cdot)$ will be a great help. Throughout the rest of this article, we assume that for all $x\in \Rxinit$, the measure $p(x,\cdot)$ is absolutely continuous with respect to the Lebesgue measure on $\R^d$. 
This is the only ambient assumption of the present paper.
By Lemma~\ref{lem: measurable density} in the appendix, there exists a measurable\footnote{While the measurability of the function $\rho(x,\cdot)$ for all $x$ is automatic, the joint measurability of $\rho$ is a genuinely stronger property whose validity requires a technical argument, that is, Lemma~\ref{lem: measurable density}.} function $\rho:(\Rxinit)\times \R^d\to [0,\infty)$ such that
\begin{equation*}
	p(x,A)=\int_A\rho(x,y)\leb (\d y),\qquad x\in \Rxinit,\ A\subseteq \R^d\ \hbox{a Borel set}.
\end{equation*} 
Here, $\leb(\cdot)$ denotes the $d$-dimensional Lebesgue measure.
It is possible to extend the definition of $\rho$ to $\Wc$. For all $x\in \Rxinit$ and $n\geq 1$, the function
\begin{equation*}
	u\mapsto \rho(x,u):=\rho(x,u_1)\prod_{i=1}^{|u|-1}\rho(u_i,u_{i+1})
\end{equation*}
is the density of $p(x,\cdot )$ with respect to the Lebesgue measure on $\Wc_n$ (which, by definition of $\Wc_n$, is simply the Lebesgue measure of $\R^{dn}$, still denoted by $\leb(\cdot)$). 
We let $\W$ denote the set of $u\in \Wc$ such that $\rho(\xinit,u)>0$. The definitions of $\W_n$ and $\W_{\leq n}$ follow similarly.

Let $k\geq 1$. The definition of $p(x,\cdot)$ over $\Wc_k$ allows us to define the iterated kernels $p^k$ and $\tilde p$ by $p^1=p$ and, for $k\geq 2$,
\begin{equation*}
	\begin{split}
	p^k(x,\d y)&=\int_{\Wc_{k-1}}p(\xi_{k-1},\d y)p(x,\d \xi),
	\\
	\tilde p(x,\d y)&=\sum_{k=1}^\infty2^{-k}p^k(x,\d y),
	\qquad \qquad \qquad x\in \Rxinit.
	\end{split}
\end{equation*}
We can easily extend the definition of $p^k(x,\cdot)$ as a probability measure on $\Wc_n$ by replacing $p(x,\d u_1)$ by $p^k(x,\d u_1)$ in \eqref{eq: p on words}. We also set $p^k(x,\cdot)=\delta_e$ on $\Wc_0$ for all $k$.
The kernels $p^k$ and $\tilde p$ have densities; for a fixed $x\in \Rxinit$,
the measurable functions
\begin{equation*}
	\begin{split}
		y\mapsto\rho^k(x,y):=\int_{\Wc_{k-1}}p(x,\d \xi)\rho(\xi_{k-1},y)
		, \qquad 
		y\mapsto \tilde \rho(x,y):=\sum_{k=1}^\infty 2^{-k}\rho ^k(x,y),
	\end{split}
\end{equation*} 
are the respective densities of the measures $p^k(x,\cdot)$ and $\tilde p (x,\cdot)$ with respect to the Lebesgue measure on $\R^d$. Replacing $y$ by some $u\in \Wc_n$ in the definition of $\rho^k$ defines $\rho^k(x,\cdot)$ as the density of $p ^k(x,\cdot)\in \Pc(\W_n)$.

\subsubsection{Communicating classes}
\label{section: classes}
We define communicating classes. This construction is possible under the mere existence of $\tilde \rho$, but its relevance will become clear with Assumption~\ref{hyp: pseudo lsc density} stated in Section~\ref{section: assumptions} below.

We call communicating class every maximal set $A$ such that $\forall x,y\in A,\ \tilde \rho(x,y)>0$.\footnote{Such maximal sets are well defined thanks to Zorn's Lemma.}
By definition, an element $x\in \R^d$ belongs to a communicating class if and only if $\tilde \rho(x,x)>0$.
We denote by $(C_j)_{j\in \Jc}$ the family of all communicating classes and we set
\begin{equation}
	\label{eq: C}
	C=\bigcup_{j\in \Jc}C_j.
\end{equation}
Under Assumption~\ref{hyp: pseudo lsc density}, the sets $C_j$ are open and do not overlap, hence the index set $\mathcal J$ is finite or countable.
We define the relation $(\leadsto)$ on $\Jc$ by
\begin{equation*}
	j_1\leadsto j_2\ \Leftrightarrow \ \forall x\in C_{j_1},\  \tilde p(x,C_{j_2})>0.
\end{equation*}
Under Assumption~\ref{hyp: pseudo lsc density}, this relation is a partial order on $\Jc$. 
By Lemma~\ref{lem: positive density}, the definition of $(\leadsto)$ is equivalent to
\begin{equation*}
	j_1\leadsto j_2\ \Leftrightarrow \ \exists x\in C_{j_1},\  \tilde p(x,C_{j_2})>0.
\end{equation*}
In addition, we write $\beta\leadsto j$ if $\tilde p(\xinit, C_j)>0$ for notational convenience.

\subsection{Assumptions}
\label{section: assumptions}
\subsubsection{Semicontinuity}
\begin{hyp}
	\label{hyp: pseudo lsc density}
	{The set} $\rho^{-1}((0,\infty))$
	is an open set of $(\Rxinit)\times \R^d$.
\end{hyp}
Of course, Assumption~\ref{hyp: pseudo lsc density} is a consequence of the simpler assumption that $\rho$ is lower semicontinuous over $(\Rxinit) \times \R^d$.
In the present article, we only use the semicontinuity assumption when establishing other assumptions. Instead, the results of the article are stated under~\ref{hyp: pseudo lsc density}. We now state some convenient properties of transitivity under~\ref{hyp: pseudo lsc density}. In particular, Lemma~\ref{lem: positive density} implies that $(\leadsto)$ is a partial order on $\Jc$.

\begin{lem}
	\label{lem: positive density}
	Assume~\ref{hyp: pseudo lsc density}. 
	\begin{enumerate}
		\item 
		The sets $(\rho^k)^{-1}((0,\infty))$ are open subsets of $(\Rxinit)\times \R^d$ for all $k\in \N$.
		\item The set $\tilde \rho^{-1}((0,\infty))$ is an open subset of $(\Rxinit)\times \R^d$.
		\item Let $x\in \Rxinit$, $y\in \R^d$, and $u\in \Wc$. If $\rho(x,uy)>0$, then $\rho^{k}(x,y)>0$, where $k=|u|+1$.
		\item 
		Let $x\in \Rxinit$, and $y,z\in \R^d$.
		If $\rho^k(x,y)>0$ for some $k\in \N$ and $\rho^l(y,z)>0$ for some $l\in\N$, then $\rho^{k+l}(x,z)>0$.
	In particular,  
	\begin{equation*}
		\big(\tilde \rho(x,y)>0,\ \tilde \rho (y,z)>0\big)\Rightarrow\tilde \rho(x,z)>0.
	\end{equation*}
\end{enumerate} 
\end{lem}
\begin{proof}
	\begin{enumerate} 
	\item 
	Observe by performing the change of variables $\xi=\xi' z$ in the integral defining $\rho ^{k+1}$ that the functions $\rho^k$ satisfy the recursive relation 
	\begin{equation}
		\label{eqloc: recursive relation for rhok}
		\rho^{k+1}(x,y)=\int_{\R^d}\rho^{k}(x,z)\rho(z,y)\leb(\d z).
	\end{equation}
	Let $(x,y)\in (\Rxinit)\times \R^d$ be such that $\rho^{k+1}(x,y)>0$. Thus, the integrand in~\eqref{eqloc: recursive relation for rhok} cannot be almost everywhere zero; there exists $z\in \R^d$ such that $\rho^{k}(x,z)\rho(z,y)>0$. Therefore, by Assumption~\ref{hyp: pseudo lsc density}, $\rho$ is positive on an open neighborhood $A_1\times A_2$ of $(z,y)$.
	In addition, if $\rho^k$ is such that $(\rho ^k)^{-1}((0,\infty))$ is open, then $\rho^{k}$ is positive on an open neighborhood $B_1\times B_2$ of $(x,z)$. For all $(x',y')\in B_1\times A_2$, since $B_2\cap A_1$ is an open set on which $\rho^{k}(x',z)\rho(z,y')>0$, the relation~\eqref{eqloc: recursive relation for rhok} yields $\rho^{k+1}(x',y')>0$. This proves that $(\rho ^{k+1})^{-1}((0,\infty))$ is an open set, hence by induction $(\rho ^k)^{-1}((0,\infty))$ is an open set for all $k\in \N$.
	\item Since $\tilde \rho\geq 2^{-k}\rho ^k$ for all $k$, the first point implies that $\tilde \rho^{-1}((0,\infty))$ is an open set of $(\Rxinit)\times \R^d$.
	\item 
	When $k=1$, we have $\rho^1(x,y)=\rho(x,y)>0$. Let $k\geq 2$.
	By definition,
	\begin{equation*}
		\rho^k(x,y)=\int_{\Wc_{k-1}}\rho(x,\xi_1)\rho(\xi_1,\xi_2)\ldots \rho(\xi_{k-2},\xi_{k-1})\rho(\xi_{k-1},y)\leb(\d \xi).
	\end{equation*}
	The integrand is positive at $\xi=u$. By Assumption~\ref{hyp: pseudo lsc density}, there exist open neighborhoods $A_i$ of each letter $u_i$ such that the integrand is positive on $A_1\times\ldots\times A_{k-1}$. Hence, the integral is positive.
	\item
	By the change of variables $\xi= \xi' y'\xi''$ in the integral defining $\rho^{k+l}(x,z)$, we have
	\begin{align*}
		\rho^{k+l}(x,z)
		&=\int_{\Wc_{k-1}}\int_{\R^d}\int_{\Wc_{l-1}}p(x,\d \xi')p(\xi_{k-1}',\d y ')p(y' ,\d \xi'')\rho(\xi''_{l-1},z)
		\\&=\int_{\R^d}\rho^{k}(x,y')\rho^l(y',z)\leb(\d y ').
	\end{align*}
	By the first point, the integrand is positive on an open neighborhood of $y$, hence the integral is positive.
	\end{enumerate}
\end{proof}
\subsubsection{Pseudo-uniformity}
If $K$ is a compact subset of $C$, we denote by $K_j$ the compact set $C_j\cap K$, and by $K_\beta$ the singleton $\{\xinit\}$. We also set, for $j\in \Jc_K:=\{\beta\}\cup\{j\in \Jc\ |\ K_j\neq \emptyset\}$,\footnote{Recall that we write $\beta\leadsto j$ when $\tilde p(\xinit, C_j)>0$.}
\begin{equation*}
	K_{ j}^+=\bigcup_{\substack{j'\in \Jc\\j\leadsto j'}}K_{j'},\qquad K_j^-=\bigcup_{\substack{j'\in \Jc\\j'\leadsto j}}K_{j'}.
\end{equation*}
We say that a compact set $K\subseteq C$ is \emph{admissible} if $\Jc_K$ is finite and totally ordered.
\begin{hyp}
	\label{hyp: pseudo uniformity with tau}
	\label{hyp: pseudo uniformity}
	For all admissible compact sets $K\subseteq C$,
	there exist a constant $c_K$, an integer $\tau_K$ and a measurable function 
	\begin{equation*}
		\tau:\bigcup_{j\in \Jc_K}K_j^-\times K_{j}\to \{1,\ldots, \tau_{K}\}
	\end{equation*}
	such that, for all $j\in \Jc_K$,  
	\begin{equation}
		\label{eq: existence of tau}
		\sup_{x_1\in K_{j}^-}\sup_{k\in \N}\rho ^k(x_1,y)\leq c_K\rho^{\tau(x_2,y)}(x_2,y),
		\qquad (x_2,y)\in K_{j}^-\times K_j.
	\end{equation}
\end{hyp}
We notice that, if~\ref{hyp: pseudo uniformity} is satisfied, then the bound~\eqref{eq: existence of tau} also holds when the letter $y\in K_j$ is replaced by a word $u\in \Wc$ whose first letter belongs to $K_j$.
 
Assumption~\ref{hyp: pseudo uniformity} is a far weaker assumption than the usual uniformity assumption~\ref{assumption U}. In order to compare them, Proposition~\ref{lem: existence of tau} below provides an alternative formulation of~\ref{hyp: pseudo uniformity}.
Assumption~{\rm (B')} of Proposition~\ref{lem: existence of tau} and Assumption~\ref{assumption U} differ in two critical ways.
First, the bound~\eqref{eq: pseudo uniformity} is not required to hold uniformly on $C\times C$, but only on compact sets.
Second, the bound~\eqref{eq: pseudo uniformity} is only required to hold for suitable $x_1,x_2$: those from which $y$ is reachable. In particular, this assumption may be satisfied in non-irreducible contexts, whereas~\ref{assumption U} obviously cannot. See the examples in Appendix~\ref{section: examples}. 
\begin{prop}
	\label{lem: existence of tau}
	\label{hyp: pseudo uniformity with density}
	Assumption~\ref{hyp: pseudo uniformity with tau} is equivalent to:	
	\\Assumption {\rm (B')}. For all admissible compact sets $K\subseteq C$, there exist a constant $c'_K$ and an integer $\tau_K$ such that, for all $j\in \Jc_K$,
	\begin{equation}
		\label{eq: pseudo uniformity}
		\rho^k(x_1,y)
		\leq 
		\frac {c'_K}{\tau_K}\sum_{i=1}^{\tau_K}\rho^{i}(x_2,y)
		, \qquad x_1,x_2\in K_j^-,\ y\in K_j,\ k\in \N.
	\end{equation}
\end{prop}
\begin{proof}
	Assume~\ref{hyp: pseudo uniformity with tau}.  For all $j\in \Jc_K$, for every $x_2\in K_j^-$ and almost every $y\in K_j$, we have
	\begin{equation*}
		\sup_{x_1\in K_j^-}\sup_{k\in \N}\rho^k(x_1,y)
		\leq c_K\rho^{\tau(x_2,y)}(x_2,y)
		\leq c_K\sum_{i=1}^{\tau_K}\rho^{i}(x_2, y ),
	\end{equation*}
	which implies \rm{(B')} with $c_K'=c_K\tau_K$.
	We now assume \rm{(B')} and prove~\ref{hyp: pseudo uniformity with tau}.
	For all $j\in \Jc_K$, for every $x\in K_j^-$ and almost every $y\in K_j$, we have
	\begin{equation*}
		\sup_{x_1\in K_{j}^-}\sup_{k\in \N}\rho^k(x_1,y)
		\leq 
		\frac {c'_K}{\tau_K}\sum_{i=1}^{\tau_K}\rho^{i}(x, y ).
	\end{equation*}
	Let $\tau$ be the function defined for every $x\in K_j^-$ and almost every $y\in K_j$ as the exponent $i$ maximizing $\rho^i(x,y)$, or the smallest of such if there are several:
	\begin{equation*}
		\tau(x,y)=\min\big\{i\in \{1,\ldots, \tau_K\}\ \big|\ \forall i'\in \{1,\ldots \tau_K\},\ \rho^{i}(x,y)\geq \rho^{i'}(x,y)\big\}.
	\end{equation*}
	This defines a measurable function $\tau$ almost everywhere on $\bigcup_{j\in \Jc_K}K_j^-\times K_{j}$.
	This function $\tau$ is such that for all $j\in \Jc_K$, for every $x_2\in K_j^-$ and almost every $y\in K_j$,
	\begin{equation*}
		\sup_{x_1\in K_{j}^-}\sup_{k\in \N}\rho^k(x_1,y)
		\leq \frac {c'_K}{\tau_K}\sum_{i=1}^{\tau_K}\rho^{i}(x_2, y )
		\leq 
		c'_K	\rho^{\tau(x_2,y)}(x_2,y).
	\end{equation*}
	Therefore,~\ref{hyp: pseudo uniformity with tau} is satisfied with $c_K=c'_K$.
\end{proof}
Showing that~\ref{hyp: pseudo uniformity} holds may seem convoluted. However, it is often quite simple in practice, as~\ref{hyp: pseudo uniformity} is actually a consequence of the boundedness and semicontinuity of $\rho$.
\begin{lem}
	\label{lem: boundedness and semicontinuity}
	If $\rho$ is lower semicontinuous and bounded on $(\Rxinit )\times \R^d$, then~\ref{hyp: pseudo uniformity with tau} is satisfied.
\end{lem}
\begin{proof} 
	Let $j,j'\in \Jc_K$ be such that $j'\leadsto j$. For all $(x,y)\in K_{j'}\times K_{j}$, there exists $l=l(x,y)\in \N$ such that $ \rho^l(x,y)>0$. 
	Indeed, by definition of the order, there exists some $z\in C_{j}$ such that $\tilde \rho(x,z)>0$; since we also have $\tilde \rho(z,y)>0$, we get $\tilde \rho(x,y)>0$, which in turn implies the existence of $l$. Therefore,
	\begin{equation*}
		K_{j'}\times K_{j}\subseteq \bigcup_{l\in \N}\bigcup_{n\in \N}\big\{(x,y)\in K_{j'}\times K_{j} \ |\ \rho^l(x,y)>1/n\big\}.
	\end{equation*}
	The set on the left-hand side of this inclusion is compact. By Fatou's Lemma, the lower semicontinuity of $\rho^l$ follows from that of $\rho$.
	Hence, the sets appearing on the right-hand side of the inclusion are open.  We extract a finite cover from this open cover; by taking $m$ as the smallest $1/n$ appearing in the finite cover, we deduce
	\begin{equation*}
		K_{j'}\times K_{j}\subseteq\bigcup_{i=1}^NB_i,\qquad B_i=\big\{(x,y)\in K_{j'}\times K_{j} \ |\ \rho^{l_i}(x,y)\geq m\big\}.
	\end{equation*}
	We set $\tau(x,y)=l_i$, for some $i$ satisfying $(x,y)\in B_i$. For all  $(x,y)\in K_{j'}\times K_j$, we have
	\begin{equation}
		\label{eqloc: rho tau}
		\rho^{\tau(x,y)}(x,y)\geq m.
	\end{equation} 
	This construction holds for any $j,j' \in \Jc_K$ such that $j'\leadsto j$, thus we have defined a function $\tau$ satisfying~\eqref{eqloc: rho tau} on $\bigcup_{j\in \Jc_K}K_j^-\times K_{j}$.
	By the boundedness assumption on $\rho$, we consider a constant $M$ such that $\rho\leq M$ on $ (\Rxinit)\times \R^d$. Then, $\rho^k\leq M$ on $ (\Rxinit)\times \R^d$ for all $k\in \N$, thus
	\begin{equation*}
		\sup_{x_1\in K_{j}^-}\sup_{k\in \N}\rho ^k(x_1,y)\leq M, \qquad y\in K.
	\end{equation*}
	Taking $c_K=M/m$ yields~\eqref{eq: existence of tau}.
\end{proof}
\subsubsection{Measure of the boundaries}
\begin{hyp}
	\label{hyp: null border}
	The boundary $\partial C$ of $C$ has Lebesgue measure $0$.
\end{hyp}
We view this assumption as mild. One simple counter-example is when $\rho(x,\cdot)$ does not depend on $x$ (thus the Markov chain is actually an i.i.d.~sequence) and is supported on the complement of the fat Cantor set. In that case, there is exactly one communicating class $C$, which is the complement of the fat Cantor set itself, hence $\partial C$ has positive Lebesgue measure.
\subsubsection{Behavior outside of the communicating classes}
\begin{hyp}
\label{hyp: getting fast to classes}
For all $\mu\in \Pc(\R^d)$ such that $\mu(C)<1$, there exists 
a neighborhood $V$ of $\mu$ (in the weak topology of $\Pc(\R^d)$) such that 
\begin{equation}
	\label{eq: getting fast to classes}
	\lim_{n\to \infty}\frac 1n\log\P(L_n\in V)=-\infty.
\end{equation}
\end{hyp}
Assumption~\ref{hyp: getting fast to classes} may seem tailor-made for the conclusions of Theorem~\ref{theo: main result}. 
This assumption allows us to avoid addressing the problematic case of measures that are not supported on $\overline C$.
However, we believe that making such an assumption is not restrictive in practice, as it is very natural for Markov chains to satisfy it. Of course, when $\overline C=\R^d$, there is nothing to prove. But even in more complicated practical cases,~\ref{hyp: getting fast to classes} is often easily proved. 
Examples~\ref{ex: rorw} and~\ref{ex: one dimensional} in Appendix~\ref{section: examples} present some arguments that can be used to prove~\ref{hyp: getting fast to classes} in dimension $d=1$.
In fact, we were unable to find any simple instance of a Markov chain whose density is both bounded and lower semicontinuous on $(\Rxinit)\times \R^d$, that satisfies~\ref{hyp: null border} while violating~\ref{hyp: getting fast to classes}.
\subsection{Main results}
\label{section: mai results}
Before stating the result, let us introduce the \emph{admissible measures}. Intuitively, admissible measures are the elements of $\Pc(\R^d)$ that can be approximated by $L_n$ for arbitrarily large $n$, with positive probability. Their definition is similar to that of admissible measures in the discrete case \cite{wu2005,daures2025}.
Let us recall that $p$ is such that $p(x,\cdot)$ is assumed to be absolutely continuous with respect to the Lebesgue measure for all $x\in \Rxinit$. The communicating classes $(C_j)$ are defined in Section~\ref{section: classes} and Assumptions ~\ref{hyp: pseudo lsc density},~\ref{hyp: pseudo uniformity},~\ref{hyp: null border} and~\ref{hyp: getting fast to classes} are defined in Section~\ref{section: assumptions} above.
\begin{defi}
	\label{defi: admissible measure}
	A measure $\mu \in \Pc(\R^d)$ is said to be \emph{admissible} if 
	\begin{enumerate}
		\item $\mu$ is absolutely continuous with respect to $\leb$;
		\item $\supp\mu\subseteq \overline C$;
		\item for all $j\in \Jc$ such that $\mu(C_j)>0$, we have $\beta\leadsto j$;
		\item the relation $\leadsto$ is a total order on $\{j\in \Jc\ |\ \mu(C_j)>0\}$;
	\end{enumerate}
	The set of admissible measures is denoted by $\Ac\subseteq \Pc(\R^d)$.
\end{defi}
Notice that
the set of probability measures satisfying the first three conditions is convex.
The set of probability measures satisfying the fourth condition may not be convex.
Nevertheless, if the fourth condition is satisfied by two measures $\mu_1$ and $\mu_2$ and by a third measure in $[\mu_1,\mu_2]\setminus \{\mu_1,\mu_2\}$, then it is satisfied by every measure of the line segment $[\mu_1,\mu_2]$. This is because the set $\{j\in \Jc\ |\ \lambda\mu_1(C_j)+(1-\lambda)\mu_2(C_j)>0\}$ is the same for all $\lambda\in(0,1)$. 
Therefore, the set $\Ac$ may not be convex, but it satisfies the same property.

We can now state the main theorem of the present paper.
\begin{theo}
	\label{theo: main result}
	Assume~\ref{hyp: pseudo lsc density},~\ref{hyp: pseudo uniformity},~\ref{hyp: null border} and~\ref{hyp: getting fast to classes}. Then, 
	$(L_n)$ satisfies the weak LDP with rate function $I$ satisfying 
	\begin{enumerate}
		\item if $\mu\notin \Ac$, then \begin{equation}
			I(\mu)=\infty;
		\end{equation}
		\item if $\mu_1,\mu_2\in \Ac$ and $\lambda\in [0,1]$ are such that $\lambda\mu_1+(1-\lambda)\mu_2\in \Ac$, then
		\begin{equation}
			\label{eq: convexity}
			I(\lambda\mu_1+(1-\lambda)\mu_2)\leq \lambda I(\mu_1)+(1-\lambda)I(\mu_2);
		\end{equation} 
		\item if $\mu_1,\mu_2\in \Ac$ are such that $\lambda\mu_1+(1-\lambda)\mu_2\in \Ac$ for some $\lambda\in [0,1]$ and if no class $C_j$ satisfies $\mu_1(C_j)\mu_2(C_j)>0$, then
		\begin{equation}
			\label{eq: linearity}
			I(\lambda\mu_1+(1-\lambda)\mu_2)= \lambda I(\mu_1)+(1-\lambda)I(\mu_2).
		\end{equation} 
	\end{enumerate}
\end{theo}
Considering Lemma~\ref{lem: boundedness and semicontinuity} and the discussions of Section~\ref{section: assumptions}, a direct corollary of this theorem is the following.
\begin{theo}
	\label{theo: main result (simplified)}
	Assume that $\rho$ is lower semicontinuous and bounded on $(\Rxinit)\times \R^d$. In addition, assume~\ref{hyp: null border} and~\ref{hyp: getting fast to classes}.
	Then, the conclusions of Theorem~\ref{theo: main result} hold.
\end{theo}
This is to be compared with existing results for discrete non-irreducible Markov chains from \cite{wu2005,rassoul,daures2025}. In particular, the set of admissible measures $\Ac$ plays the same role and we recover the same properties of $I$ as in the discrete case.
We believe that, as in the discrete case, $I$ coincides on $\Ac$ with the Donsker-Varadhan entropy. However, the proof of this \emph{a posteriori} identification by \cite{daures2025} cannot accommodate the continuous case, as it involves a \emph{minimal cycle decomposition}, which is a purely discrete object.

The rest of this article is dedicated to proving Theorem~\ref{theo: main result}. The theorem follows from Propositions~\ref{prop: IRL for non admissible measures},~\ref{prop: IRL for admissible measures} and~\ref{prop: IRL linearity}, that are stated in Section~\ref{section: Ruelle Lanford} below and proved in Sections~\ref{section: non admissible},~\ref{section: coupling and decoupling} and~\ref{section: decoupling}; this is where the subadditive method is actually performed.
Once Theorem~\ref{theo: main result} is proved, Theorem~\ref{theo: main result (simplified)} follows by Lemma~\ref{lem: boundedness and semicontinuity}.

\section{The Ruelle-Lanford function}
\label{section: Ruelle Lanford}
The weak topology of $\Pc(\R^d)$ is metrized by the L\'evy-Prokhorov metric $\dlp$ (see Appendix~\ref{section: levy prokhorov}). Recall that all $\mu,\nu\in \Pc(\R^d)$ satisfy $\dlp(\mu, \nu)\leq\dtv(\mu, \nu):= |\mu-\nu|_\TV$, where $|\cdot|_\TV$ denotes the total variation norm on the set of finite signed measures of $\R^d$. The open L\'evy-Prokhorov ball of radius $\delta$ and center $\mu$ is denoted by $\Bc_{\LP}(\mu, \delta)$. 
\begin{defi}
	For every measurable set $A$ of $\Pc(\R^d)$, we set 
	\begin{align*}
		\underline s(A)&=\liminf_{n\to \infty}\frac 1n \log \P(L_n\in A),\\
		\overline s(A)&=\limsup_{n\to \infty}\frac 1n \log \P(L_n\in A).
	\end{align*}
	Let $\mu\in \Pc(\R^d)$.
	We set \begin{align*}
		\underline{s}(\mu)&=\lim_{\delta\to 0}\underline s(\Bc_{\LP}(\mu,\delta)),\\
		\overline{s}(\mu)&=\lim_{\delta\to 0}\overline s(\Bc_{\LP}(\mu,\delta)).
	\end{align*}
	If $\underline s(\mu)=\overline s(\mu)$, we set $\IRL(\mu)=-\underline{s}(\mu)=-\overline s(\mu)$ and we say that the \emph{Ruelle-Lanford function} of $(L_n)$ exists at $\mu$.
\end{defi}

Ruelle-Lanford functions are named after Ruelle and Lanford, who used a subadditive approach to define and compute limits of logarithmic moment generating functions \cite{ruelle1965,ruelle1967,lanford1973} before they were used in large deviations. See \cite{pfisterlewis1995} for a detailed historical account. Our use of the Ruelle-Lanford function lies in the following standard lemma.
\begin{lem}
	\label{lem: RL function implies LDP}
	If the Ruelle-Lanford function of $(L_n)$ exists at all $\mu\in \Pc(\R^d)$, then $\IRL$ is lower semicontinuous on $\Pc(\R^d)$ and $(L_n)$ satisfies the weak LDP with rate function $I=\IRL$.
\end{lem}
For a (simple) proof of this lemma and related properties, see \cite[Section 4.1.2]{DZ}.
This is the cornerstone of the subadditive method. 
To prove the weak LDP of Theorem~\ref{theo: main result}, it suffices to prove the existence of the Ruelle-Lanford function.
This is achieved by deriving supermultiplicative inequalities for some sequences $(\P(L_n\in \Bc_\LP(\mu,\delta)))$\footnote{Thus, \emph{subadditive} inequalities for the sequence $(-\log\P(L_n\in \Bc_\LP(\mu,\delta)))$.} in the next sections.
The additional properties of $I$ stated in Theorem~\ref{theo: main result} are properties of $\IRL$ that come with the proof of its existence. We state these properties here, and prove them in the next sections.

\begin{prop}
	\label{prop: IRL for non admissible measures}
	Assume~\ref{hyp: pseudo lsc density} and~\ref{hyp: getting fast to classes}. Let $\mu\in \Pc(\R^d)\setminus \Ac$. Then, the Ruelle-Lanford function of $(L_n)$ exists at $\mu$ and $\IRL(\mu)=\infty$.
\end{prop}
\begin{prop}
	\label{prop: IRL for admissible measures}
	Assume~\ref{hyp: pseudo lsc density},~\ref{hyp: pseudo uniformity} and~\ref{hyp: null border}. Then, the Ruelle-Lanford function $\IRL$ of $(L_n)$ exists at all $\mu\in \Ac$. Moreover, for all $\mu_1,\mu_2\in\Ac$ and $\lambda\in [0,1]$ such that $\lambda\mu_1+(1-\lambda)\mu_2\in \Ac$,
	\begin{equation}
		\label{eq: convexity of IRL}
		\IRL(\lambda\mu_1+(1-\lambda)\mu_2)\leq \lambda\IRL(\mu_1)+(1-\lambda)\IRL(\mu_2).
	\end{equation} 
\end{prop}
\begin{prop}
	\label{prop: IRL linearity}
	Assume~\ref{hyp: pseudo lsc density},~\ref{hyp: pseudo uniformity} and~\ref{hyp: null border}. 
	For all $\mu_1,\mu_2\in\Ac$ and $\lambda\in [0,1]$ such that $\lambda\mu_1+(1-\lambda)\mu_2\in \Ac$, if no communicating class $C_j$ satisfies $\mu_1(C_j)\mu_2(C_j)>0$, then
	\begin{equation}
		\label{eq: linearity 2}
		\IRL(\lambda\mu_1+(1-\lambda)\mu_2)\geq\lambda\IRL(\mu_1)+(1-\lambda)\IRL(\mu_2).
	\end{equation}
\end{prop}
Combining Propositions~\ref{prop: IRL for non admissible measures},~\ref{prop: IRL for admissible measures} and~\ref{prop: IRL linearity} yields Theorem~\ref{theo: main result}. 
Note that Propositions~\ref{prop: IRL for non admissible measures} and~\ref{prop: IRL for admissible measures} alone are already sufficient to establish the weak LDP, while Proposition~\ref{prop: IRL linearity} is only needed to derive the last property of the rate function.
We prove Proposition~\ref{prop: IRL for non admissible measures} in Section~\ref{section: non admissible}.
We prove Proposition~\ref{prop: IRL for admissible measures} in Section~\ref{section: coupling and decoupling}, using the tools provided in Section~\ref{section: slicing and stitching}.
We prove Proposition~\ref{prop: IRL linearity} in Section~\ref{section: decoupling}, using the tools provided in Section~\ref{section: slicing and stitching}.

\subsection{The Ruelle-Lanford function for non-admissible measures}
\label{section: non admissible}
\begin{prop}
	\label{prop: nonadmissible: abscont}
	Let $\mu\in \Pc(\R^d)$. Suppose that $\mu$ is not absolutely continuous with respect to the Lebesgue measure. Then, $\IRL(\mu)$ exists and $\IRL(\mu) =\infty$.
\end{prop}
\begin{proof}
	Let $\IDV$ denote the Donsker-Varadhan entropy:
	\begin{equation}
		\label{eq: DV entropy}
		\IDV(\mu)=\sup_{f}\int_{\R^d}\log \frac{f(x)}{pf(x)}\mu(\d x),
	\end{equation}
	where the supremum is taken over the set of measurable functions $f:\R^d\to (0,\infty)$ satisfying $\epsilon \leq f\leq 1/\epsilon$ for some $\epsilon >0$. It is well known 
	\cite[Theorems 3.2, 4.1]{deacosta2022} that the large deviations weak upper bound is always satisfied with rate function $\IDV$. 
	Hence \cite[Lemma 4.1.24]{DZ} yields the following implication: if $\IDV(\mu)=\infty$, then $\overline s(\mu)=-\infty$, further implying that $\IRL(\mu)$ exists and is infinite.
	Let us show that $\IDV(\mu)=\infty$.
	Let $A$ be a Borel set such that $\mu(A)>0$ and $\leb(A)=0$. Let $a>0$ and $f(x)=a\mathbf 1_A(x)+1$. This defines a function that is measurable, bounded, and bounded away from $0$. We have
	\begin{equation*}
		pf(x)=a\int_A\rho(x,y)\leb(\d y)+1=1.
	\end{equation*}
	Thus,
	\begin{equation*}
		\int_{\R^d}\log \frac {f(x)}{pf(x)}\mu(\d x)=\int_{\R^d}\log f(x)\mu (\d x)=\mu(A)\log (1+a),
	\end{equation*}
	which can be arbitrarily large. This shows that the quantity in the supremum of~\eqref{eq: DV entropy} can be arbitrarily large. Hence $\IDV(\mu)=\infty$, which concludes the proof.
\end{proof}
\begin{prop}
\label{prop: nonadmissible: support}
Assume~\ref{hyp: getting fast to classes}. Let $\mu\in \Pc(\R^d)$ be such that $\supp\mu\not \subseteq \overline \C$. Then, $\IRL(\mu)$ exists and $\IRL(\mu)=\infty$. 
\end{prop}
\begin{proof}
Since $\supp\mu\not \subseteq \overline \C$, we have $\mu(C)<1$, hence Assumption~\ref{hyp: getting fast to classes} applies.
By definition of $\IRL$, we have $\IRL(\mu)=\infty$.
\end{proof}
\begin{prop}
	\label{prop: nonadmissible: reachability from 0}
Assume~\ref{hyp: pseudo lsc density}. Let $\mu\in \Pc(\R^d)$ be absolutely continuous with respect to the Lebesgue measure. Assume that $\mu(C_{j})>0$ for some $j\in\Jc$ satisfying $\beta\not\leadsto j$. Then, $\IRL(\mu)$ exists and $\IRL(\mu)=\infty$.
\end{prop}
\begin{proof}
	We prove the contrapositive. Assume $\overline s(\mu)>-\infty$.
	Let $\delta$ be as in Lemma~\ref{lem: dlp small implies similar support} with $O=C_j$. By definition of $\underline s$, there must exist a word $u\in \W$ such that $L[u]\in \Bc_\LP(\mu, \delta)$. 
	By Lemma~\ref{lem: dlp small implies similar support}, $L[u](C_j)>0$, thus at least one letter of $u$ belongs to $C_j$. Up to taking a prefix of $u$ instead of $u$, we can assume that the last letter of $u$, denoted by $x$, belongs to $C_j$. By Lemma~\ref{lem: positive density}, we have
	$\rho^k(\xinit,x)>0$, where $k=|u|$. Assumption~\ref{hyp: pseudo lsc density} further yields $p^k(\xinit, C_j)>0$; in other words, $\beta\leadsto j$.
\end{proof}
\begin{prop}
\label{prop: nonadmissible: comparison of classes}
Assume~\ref{hyp: pseudo lsc density}. Let $\mu\in \Pc(\R^d)$ be absolutely continuous with respect to the Lebesgue measure. Assume that $\mu(C_{j_1})\mu(C_{j_2})>0$ for some $j_1,j_2\in\Jc$ that are not comparable under $(\leadsto)$. Then, $\IRL(\mu)$ exists and $\IRL(\mu)=\infty$.
\end{prop}
\begin{proof}
We prove the contrapositive. Assume $\overline s(\mu)>-\infty$. 
We use Lemma~\ref{lem: dlp small implies similar support} to find $\delta>0$ such that all $\nu\in \Bc_\LP (\mu, \delta)$ satisfy $\nu(C_{j_1})\nu(C_{j_2})>0$. By definition of $\overline s$, there must exist a word $u\in \W$ such that $L[u]\in \Bc_\LP(\mu, \delta)$.
By Lemma~\ref{lem: dlp small implies similar support}, $L[u](C_{j_1})>0$ and $L[u](C_{j_2})>0$, thus both $C_{j_1}$ and $C_{j_2}$ contain at least one letter of $u$. We denote by $x_1$ and $x_2$ two such letters and we assume, for instance, that $x_1$ comes before $x_2$ in $u$. Since $u\in \W$, Lemma~\ref{lem: positive density} shows us that $\rho^{k+1}(x_1,x_2)>0$ where $k$ is the number of letters between $x_1$ and $x_2$ in $u$. 
By Assumption~\ref{hyp: pseudo lsc density}, we further have $p^{k+1}(x_1,C_{j_2})>0$, thus $j_1\leadsto j_2$. If we assumed that $x_2$ came before $x_1$ in $u$ instead, we would have obtained $j_2\leadsto j_1$. In both cases, $j_1$ and $j_2$ are comparable.
\end{proof}
\subsection{Slicing and stitching words}
In this section, we introduce the \emph{slicing}, \emph{stitching}, \emph{coupling} and \emph{decoupling} maps, which are operations on words. Like in \cite{daures2025}, these operations will allow us to recombine portions of words together in order to build a long word from several small ones.
\label{section: slicing and stitching}
Given $u\in \Wc\setminus \{e\}$, we set 
\begin{equation*}
	L[u]=\frac 1{|u|}\sum_{j=1}^{|u|}\delta_{u_j}.
\end{equation*}
Given a sequence of non-empty words\footnote{Except when explicitly specified, exponents on words do not denote powers, but simply a numbering. Indices are reserved for numbering the letters of a given word: $u^i_j$ is the $j$-th letter of the $i$-th word. 
	Underlined letters will always represent lists of words; $\underline v$ is the list $(v^1,\ldots, v^k)$.} $\underline u=(u^1,\ldots, u^k)\in (\Wc\setminus\{e\})^k$, we set
\begin{equation*}
	L[\underline u]=\sum_{i=1}^k \frac{|u^i|}{|\underline u|}L[u^i], \qquad |\underline u|:=\sum_{i=1}^k|u^i|.
\end{equation*}
Throughout this section, $K\subseteq C$ is an admissible compact set. For convenience, we relabel elements of $\Jc_K$ as $\beta,1,\ldots,r$, so that $\beta\leadsto1\leadsto\ldots \leadsto r$.
\subsubsection{The slicing map} 
\label{section: slicing map}
\begin{defi}[The slicing map]
	Let $u\in \Wc_n$. For each $1\leq j\leq r$, if $u$ has no letters in $K_j$, let $u^j=e$. Otherwise, let $u^j$ be the subword of $u$ whose first letter is the first letter of $u$ within $K_j$ and whose last letter is the last letter of $u$ within $K_j$.
	We define $\sli_n(u)=(u^1,\ldots ,u^r)\in \Wc^r$.
\end{defi}
Implicitly, the slicing map depends on the choice of the compact $K$.
Since we have a total order on $(C_j)_{1\leq j\leq r}$, every $u\in \W_n$ can be written
\begin{equation}
	u=\zeta^1(u)u^1\zeta ^2(u)u^2\ldots \zeta^{r}(u)u ^r\zeta^{r+1}(u),
\end{equation}
where $\zeta^j(u)$ are some words of length at most $n$.
See Figure~\ref{fig: slicing map} for a visual depiction of the slicing map applied to two words $u^1,u^2\in \W_n$.
In the next proposition, we formulate bounds in terms of a simple continuous function $h: (0, \infty)\to [0,\infty)$, defined by
\begin{equation}
	\label{eq: function h}
	h(x)=\Big|\frac 1x-1\Big|+|1-x|.
\end{equation}
The function $h$ will provide rough bounds that are more convenient than the precise ones obtained in the proof. Notice that $\lim_{x \to 1}h(x)=0$.
\begin{prop}
	\label{prop: geographic ineq slicing map}
	\begin{enumerate}
		\item
	Let $n\in \N$
	and let $u\in \W_n$ be such that $\underline w:=\sli_n(u)\neq (e,\ldots, e)$. Then,
	\begin{equation}
		\label{eq: geographic ineq slicing map 1}
		\dlp(L[u],L[\underline w])\leq\dtv(L[u],L[\underline w]) \leq h\Big(\frac{|\underline w|}{n}\Big).
	\end{equation}
	\item
	Let $k\in \N$ and let $\underline u:=(u^{1},u^ 2,\ldots, u^k)\in \W^k_n$ be such that $\underline w^i:=\sli_n(u^i)\neq (e,\ldots, e)$ for all $i$. Then,
	\begin{equation}
		\label{eq: geographic ineq slicing map 2}
		\dlp(L[\underline u],L[\underline v])\leq \frac 1k\sum_{i=1}^k\Big(h\Big(\frac{k|\underline w^i|}{|\underline v|}\Big)+h\Big(\frac{|\underline w^i|}{n}\Big)\Big),
	\end{equation}
	where $\underline v\in \Wc^{kr}$ is the concatenation of the lists $\underline w^1,\ldots,\underline w^k$. 
\end{enumerate}
\end{prop}
\begin{proof}
	We begin by proving~\eqref{eq: geographic ineq slicing map 1}. We have $\dlp\leq \dtv$ on $\Pc(\R^d)^2$. Moreover,
	\begin{align*}
		\dtv(L[u],L[\underline w]) 
		&=
		\bigg|L[u]- \frac 1{|\underline w|}\sum_{j=1}^r|u^j|L[u^j]\bigg|_\TV
		\leq
		1-\frac{|\underline{w}|}{n}+\frac 1n\bigg |\sum_{i=1}^n\delta_{u_i}-\sum_{j=1}^r\sum_{i=1}^{|u^j|}\delta_{u^j_i}\bigg|_\TV.
		\end{align*}
On the right-hand side, each Dirac mass in the second sum also appears in the first. Hence,
\begin{align*} 
	\dtv(L[u],L[\underline w]) 	\leq 
		1-\frac{|\underline{w}|}{n}+\frac1n(n-|\underline w|)
		&\leq h\Big(\frac{|\underline w|}{n}\Big).
	\end{align*}
	The bound with $h(|\underline w|/n)$ is due to $|\underline w|\leq n$.
	We have proved~\eqref{eq: geographic ineq slicing map 1}.
	We now prove~\eqref{eq: geographic ineq slicing map 2}. Using~\eqref{eq: geographic ineq slicing map 1}, we have
	\begin{align*}
		\dlp(L[\underline u],L[\underline v])
		\leq\dtv(L[\underline u],L[\underline v])
		&=\bigg|\sum_{i=1}^k\frac 1k L[u^i]-\sum_{i=1}^k\frac{|\underline w^{i}|}{|\underline v|}L[\underline w^i]\bigg|_\TV
		\\&\leq \sum_{i=1}^k\Big(\Big|\frac 1k -\frac{|\underline w^i|}{|\underline v|}\Big|+\frac 1k|L[u^i]-L[\underline w^i]|_{\TV}\Big)
		\\&\leq \sum_{i=1}^k\Big|\frac 1k -\frac{|\underline w^i|}{|\underline v|}\Big|+ \frac 1k\sum_{i=1}^kh\Big(\frac{|\underline w^i|}{n}\Big)
		\\&\leq \frac 1k\sum_{i=1}^k\Big(h\Big(\frac{k|\underline w^i|}{|\underline v|}\Big)+h\Big(\frac{|\underline w^i|}{n}\Big)\Big).
	\end{align*}
	We have thus proved~\eqref{eq: geographic ineq slicing map 2}.
\end{proof}
\subsubsection{The stitching map}
\label{section: stitching map}
In this section, we assume~\ref{hyp: pseudo uniformity with tau}. The function $\tau$ and the constants $\tau_K$ and $c_K$ are defined in Assumption~\ref{hyp: pseudo uniformity with tau} and depend on $K$.
In the discrete case of \cite{daures2025}, the output of the stitching map was a word. Here, we need $\sti_{k,T}(\underline v)$ to be a set of words in order to later make the inequality $\P(\sti_{k,T}(\underline v))>0$ possible.
See Figure~\ref{fig: stitching map} for a visual depiction of the stitching map.
\begin{defi}[The stitching map]
	\label{defi: stitching map}
	 Let $k\in \N$. We say that a finite sequence of words $\underline v=(v^1,\ldots, v^k)\in \Wc^k$ is stitchable if there exists
	 a nondecreasing function $j:\{1,\ldots, k\}\to\{1,\ldots, r\}$ such that 
	 for each $1\leq i\leq k$ satisfying $v^i\neq e$, the first and last letter of $v^i$ belong to $K_{j(i)}$. Let $T\in \N$. We set 
	\begin{equation*}
		\Sc_{k,T}=\{\underline v\in \Wc^k,\, \hbox{$\underline v$ is stitchable},\ |\underline v|+ k\tau_K\leq T\}.
	\end{equation*}
	Let $\underline v\in \Sc_{k,T}$. We define the measurable set $\sti_{k,T}(\underline v)\subseteq \Wc_T$ in the following way:
	for each $1\leq i\leq k$,
	\begin{enumerate}
		\item if $v^i=e$, we set $\tau_{i}=0$;
		\item otherwise, let $y=v^{i}_1$ and let $x$ be the last letter of the last non-empty $v^{j}$ before $v^i$ in the sequence (or $\xinit$ if there are none). Then, $(x,y)\in \bigcup_{j=1}^rK_j^-\times K_{j}$ (because the sequence is stitchable) and we set $\tau_{i}=\tau(x,y)$.
	\end{enumerate}
	Finally, we set $\tau_{k+1}=T-(|\underline v|+\tau_1+\ldots +\tau_k)\geq 0$ and 
	\begin{equation*}
		\sti_{k,T}(\underline v)=\Wc_{\tau_1}\times \{v^1\}\times \Wc_{\tau_2}\times \{v^2\}\times \ldots\times \Wc_{\tau_k}\times \{v^k\}\times \Wc_{\tau_{k+1}}\subseteq \Wc_T. 
	\end{equation*}
\end{defi}
\begin{prop}
	\label{prop: geographic ineq stitching map}
	Let $\underline v\in \Sc_{k,T}$ be such that $\underline v\neq(e,\ldots, e)$. Then, for all $w\in \sti_{k,T}(\underline v)$, we have
	\begin{equation*}
		\dlp (L[w],L[\underline v])\leq 2h\Big(\frac {|\underline v|}{T}\Big).
	\end{equation*}
\end{prop}
\begin{proof}
	Using that $\dlp\leq \dtv$, we obtain
	\begin{equation*}
		\dlp(L[w],L[\underline v])\leq \Big(\frac T{|\underline v|}-\frac T{T}\Big)+\frac 1{|\underline v|}\bigg|\sum_{i=1}^{|w|}\delta_{w_i}-\sum_{i=1}^k\sum_{j=1}^{|v^i|}\delta_{v^i_ j}\bigg|_{\mathrm {TV}}.
	\end{equation*}
	Since words $v^i$ are non-overlapping subwords of $w$, the two sums in the TV norm on the right-hand side differ only by Dirac terms corresponding to letters of $w$ that are not in its subwords $v^i$. There are at most $T-|\underline v|$ such letters and hence the second term is at most $(T-|\underline v|)/|\underline v|$. 
	This completes the proof because $1+1/T\leq 2$.
\end{proof}
\subsubsection{The coupling map}
\label{section: coupling map}
In this section, we assume~\ref{hyp: pseudo uniformity with tau}. The coupling map is a composition of the slicing map and the stitching map. Given $N$ words of length $n$, it provides a set of words of length $T$ involving subwords of the initial words. 
See Figure~\ref{fig: coupling map} for a visual depiction of the coupling map.
\begin{figure}[!tb]
	\centering
	\begin{subfigure}[b]{1\textwidth}
		\includegraphics[width=\textwidth]{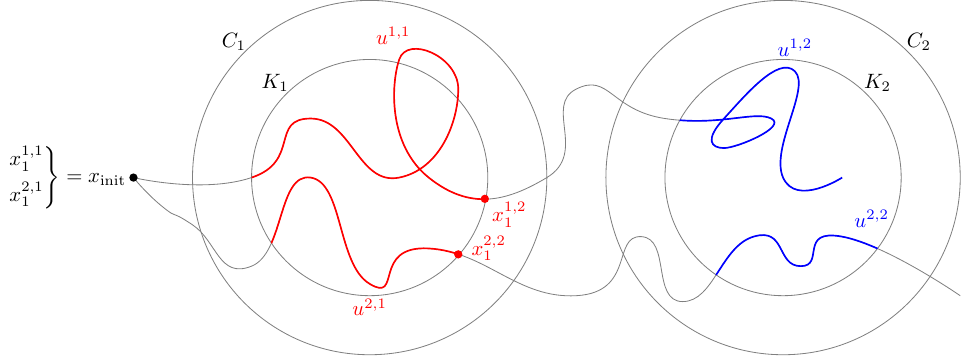}
		\captionsetup{width=0.7\textwidth}
		\captionsetup[subfigure]{skip=1em}
		\caption{The action of the slicing map on two words $u^1$ and $u^2$. The slicing map removes the gray bits of the trajectories and yields a list of four subwords $\underline v=(u^{1,1}, u^{1,2},u ^{2,1}, u^{2,2})$.
		The letters $x_1^{i,j}$ used in the proof of Proposition~\ref{prop: proba ineq coupling map} are indicated.\vspace{0.5cm}}
		\label{fig: slicing map}
	\end{subfigure}
	\begin{subfigure}[b]{1\textwidth}
		\includegraphics[width=\textwidth]{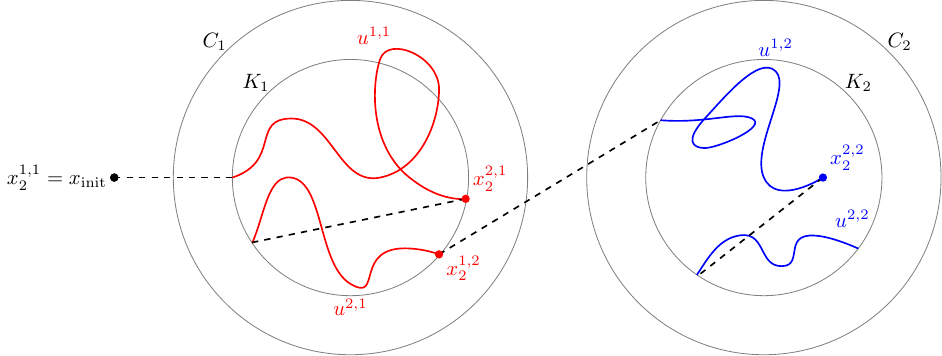}
		\captionsetup{width=0.7\textwidth}
		\caption{The action of the stitching map on the stitchable list $\sigma(\underline v)=(u^{1,1},u^{2,1},u^{1,2},u^{2,2})$. The dashed lines represent the insertion of $\Wc_{\tau_k}$ between words of $\sigma(\underline v)$. The letters $x_2^{i,j}$ used in the proof of Proposition~\ref{prop: proba ineq coupling map} are indicated.}
		\label{fig: stitching map}
	\end{subfigure}
	\captionsetup{width=0.8\textwidth}
	\caption{Example of use of the coupling map, with $N=2$ and $r=2$. Trajectories are depicted as smooth curves for easier readability, even though they are discrete in time. Given $\underline u=(u^1,u^2)$, the slicing map produces four subwords of $u^1$ and $u^2$ in~\eqref{fig: slicing map}. These subwords are reordered and then reassembled by the stitching map in~\eqref{fig: stitching map}. }
	\label{fig: coupling map}
\end{figure}
\begin{lem}
	\label{lem: reordering map}
	Let $N,n\in \N$. There exists $\sigma:\sli_n(\W_n)^N\to \Wc^{Nr}$ such that the elements of $\sigma(\underline v)$ are exactly the elements of $\underline v$ and $\sigma(\underline v)$ belongs to $\Sc_{Nr,T}$ for every integer $T\geq Nn+Nr\tau_K$.
\end{lem}
\begin{proof}
	Let $\underline u=(u^1,\ldots, u^N)\in \W^N_n$. Applying $\sli$ to each $u^i$ yields an $N\times r$ matrix of words $u ^{i,j}$. If $u^{i,j}\neq e$, then the first and last letters of $u^{i,j}$ belong to $K_j$. We set $\sigma(\sli_n(u^1),\ldots, \sli_n(u^N))$ to be the list obtained by reading the entries of the matrix $(u^{ij})$ column by column. This list contains $Nr$ elements and is stitchable. Moreover,
	\begin{equation*}
		|\sigma(\sli_n(u^1),\ldots, \sli_n(u^N))|=\sum_{1\leq i\leq N}\sum_{1\leq j\leq r}|u ^{i,j}|\leq \sum_{1\leq i\leq N}|u^i|= Nn.
	\end{equation*}
	Hence $\sigma(\sli_n(u^1),\ldots, \sli_n(u^N))\in \Sc_{Nr, T}$ if $T\geq Nn+Nr\tau_K$.
\end{proof} 
\begin{defi}[The coupling map]
	Let $N,n, T\in \N$ be such that $T\geq Nn+Nr\tau_K$.
	Let $\underline u=(u^1,\ldots ,u ^N)\in \W_n^N$. Let $\underline v=(\sli_n(u^1),\ldots, \sli_n(u^N))$. 
	Then, $\sigma(\underline v)\in \Sc_{Nr,T}$, and we set
	\begin{equation*}
		\cou_{N,n,T}(\underline u)= \sti_{Nr,T}(\sigma(\underline v))=\sti_{Nr,T}(\sigma(\sli_n(u^1),\ldots, \sli_n(u^N)))\in \Bc(\Wc_{T}).
	\end{equation*}
\end{defi}
\begin{prop}
	\label{prop: geographic ineq coupling map}
	Let $\underline u=(u^1,\ldots ,u ^N)\in \W_n^N$ be such that $\underline u^i:=\sli_n(u^i)\neq(e,\ldots,e)$ for all $i$, and $\underline v= (\underline u^1,\ldots, \underline u^N)$. Then, for all $w\in \cou_{N,n, T}(\underline u)$, we have
	\begin{equation}
		\label{eq: geographic ineq coupling map}
		\dlp(L[\underline u],L[w])\leq \frac 1N\sum_{i=1}^N\Big(h\Big(\frac{N|\underline u^i|}{|\underline v|}\Big)+h\Big(\frac{|\underline u^i|}{n}\Big)\Big)+ 2h\Big(\frac{|\underline v|}{T}\Big).
	\end{equation}
\end{prop}
\begin{proof}
	Since $L[\underline v]=L[\sigma(\underline v)]$, we have
	\begin{align*}
		\dlp(L[\underline u],L[w])\leq \dlp(L[\underline u],L[\underline v])+\dlp(L[\sigma(\underline v)],L[w]).
	\end{align*}
	It then suffices to combine the bounds of Propositions~\ref{prop: geographic ineq slicing map} and~\ref{prop: geographic ineq stitching map}. The bound~\eqref{eq: geographic ineq coupling map} follows since $|\sigma(\underline v)|=|\underline v|$.
\end{proof}
The following proposition is the core of our subadditive method. It states that the long words $w$ resulting from the coupling of small words $u^i$ are approximately as probable under $\P$ as the small words $u^i$ independently picked according to $\P$.
One can prove it without using densities, provided a bound such as in Assumption~\ref{hyp: pseudo uniformity with tau} is satisfied.
\begin{prop}
\label{prop: proba ineq coupling map}
Let $W$ be a measurable subset of $\Wc_T$, and let
\begin{equation*}
	U=\{\underline u\in \W_n^N\ |\ \cou_{N,n,T}(\underline u)\subseteq W\}.
\end{equation*}
Then,\footnote{Note that the factor in~\eqref{eq: proba ineq coupling map} is simpler than the corresponding one in \cite{daures2025}. This is because our stitching map allows empty words as input.}
\begin{equation}
	\label{eq: proba ineq coupling map}
	 \P^{\otimes N}(U)\leq\big(c_K^r(\tau_K+1)^r(n+1)^{2r+1}\big)^{N}\P (W).
\end{equation}
\end{prop}
\begin{proof}
	Let us begin by introducing some useful notation. Throughout the proof we will manipulate a generic $\underline u\in U$, on which several variables will implicitly depend.
	For $1\leq i\leq N$, we set $(u^{i,1},\ldots ,u^{i, r}):=\sli_n(u^i)$ and we define the words $\zeta^{i,j}$ by the expressions
	\begin{equation}
		\label{eq: decomposition ui}
		u^i=\zeta ^{i,1}u^{i,1}\zeta ^{i,2}u^{i,2}\ldots \zeta ^{i,r}u^{i,r}\zeta ^{i,r+1},
		\qquad 1\leq i\leq N,
	\end{equation}
	where by convention $\zeta ^{i,j}=e$ when $u^{i,j}=e$ (this choice avoids some ambiguity when some $u^{i,j}$ is empty). The words $\zeta^{i,j}$ implicitly depend on $\underline u$. 
	We also set
	\begin{equation*}
		\underline v=(\sli_n(u^1),\ldots, \sli_n (u ^N))=(u^{1,1},u^{1,2},\ldots,u^{N,r-1},u^{N,r})\in \Wc_{\leq n}^{Nr},
		\footnote{We now view the object $\underline v=(\sli_n(u^1),\ldots, \sli_n(u^N))$ as the concatenation of the lists $\sli_n(u^1),\ldots,\sli_n(u^N)$, which is a list of $Nr$ words, rather than a $N\times r$ matrix of words.}
	\end{equation*}
	which also implicitly depends on $\underline u$.
	The cornerstone of the present proof will be the use of Assumption~\ref{hyp: pseudo uniformity with tau}, which involves two letters $x_1$ and $x_2$. We now introduce letters $x_1^{i,j}$ and $x_2^{i,j}$ for later use. 
	We denote by $x_1^{i,j}$ the last letter of the last non-empty word preceding $u^{i,j}$ 
	in the list $\sli_n(u^i)$, or $\xinit$ if there are none. In particular, notice that $x^{i,1}_1=\xinit$ for all $i$.
	Similarly, we denote by $x_2^{i,j}$ the last letter of the last non-empty word preceding $u^{i,j}$ in the list $\sigma(\underline v)$, or $\xinit$ if there are none. 
	See Figure~\ref{fig: coupling map} for a visual depiction of the definitions of $x_1^{i,j}$ and $x_2^{i,j}$.
	Notice that $x_1^{i,j}\in K_{j_1}$ and $x_2^{i,j}\in K_{j_2}$ for some $j_1,j_2\leadsto j$.\footnote{Including the case when $j_1=\beta$ or $j_2=\beta$.}
	To manipulate words of fixed lengths, we define the sets
		\begin{align*}
		\mathcal E_{k }
		=\big\{\underline u \in \W_n^N\ |\ \forall 1\leq i\leq N, \, 1\leq j&\leq r+1,
		\ |\zeta ^{i,j}|=k_{i,j}\big\},
		\\
		\mathcal F_{l}
		=\big\{\underline u \in \W_n^N\ |\ \forall 1\leq i\leq N, \, 1\leq j&\leq r,
		\  |u^{i,j}|=l_{i,j}\big\},
		\\
		\mathcal G_{\tau}
		=\big\{\underline u \in \W_n^N\ |\ \forall 1\leq i\leq N, \, 1\leq j&\leq r,
		\  \tau(x_{2}^{i,j},u^{i,j}_1)=\tau_{i,j}\ \hbox{if $\tau_{i,j}>0$} \},
	\end{align*}
	where $k=(k_{i,j})$, $l=(l_{i,j})$ and $\tau=(\tau_{i,j})$ are $N\times (r+1)$ and $N\times r$ matrices of integers such that $k_{i,1}+l_{i,1}+\ldots+ k_{i,r}+l_{i,r}+k_{i, r+1}= n$ for all $i$ and $\tau_{i,j}\leq \tau_K$ for all $(i,j)$. 
	We set
	\begin{align*}
		U_{k ,l,\tau}=U\cap \mathcal E_{k}\cap \mathcal F_l\cap \mathcal G_\tau,
		\qquad 
		V_\tau=
		V\cap\sli_n^{\otimes N}(\mathcal G_\tau),
	\end{align*}
	where 
	\begin{equation*}
	V=\{\underline v'\in \Sc _{Nr, T}\ |\ \sti_{Nr, T} (\underline v')\subseteq W\}.
	\end{equation*}	 
	By definition, $\underline u\in U$ if and only if $\sigma(\underline v)\in V$ if and only if $\cou_{N,n, T}(\underline u)\subseteq W$.
	The goal is to bound  the quantity
		\begin{align}
		\label{eqloc: PN(Ukltau)}
		\P^{\otimes N}(U_{k,l,\tau})
		&=\int_{\mathcal E_k\cap \mathcal F_l} 
		\mathbf 1_{U\cap \mathcal G_\tau}(\underline u)
		\prod _{i=1}^N	p(\xinit,\d  u^i).
	\end{align}
	Here, the indicator function only involves the subwords $u^{i,j}$ and does not constrain the subwords $\zeta^{i,j}$. This means that $\mathbf 1_{U\cap \mathcal G_\tau}(\underline u)$ can be replaced by $\mathbf 1_{V_\tau}(\sigma (\underline v))$ in the integral.
	Let the symbol $z ^{i,j}$ denote the letter preceding $u^{i,j}$ in~\eqref{eq: decomposition ui}, or $\xinit$ if there are none~---~as for earlier notations, the letter $z^{i,j}$ implicitly depends on $\underline u$. These notations allow the change of variable 
	\begin{equation*}
		p(\xinit, \d u^i)=p(x_1^{i,1},\d \zeta^{i,1})p(z^{i,1},\d u^{i,1})p(x_1^{i,2},\d\zeta^{i, 2})\ldots 
		p(z^{i,r},\d u^{i,r})p(x_1^{i,r},\d\zeta^{i, r+1})
	\end{equation*}
	in integrals over $\Wc_n$ with fixed lengths of subwords.\footnote{To address the case when some $k_ {i,j}$ or some $l_{i,j}$ are null, we recall the convention that $p(x,\cdot)=\delta_e$ over $\Wc_0=\{e\}$, for all $x\in \Rxinit$.}
	 Performing these changes of variables in~\eqref{eqloc: PN(Ukltau)} yields
	\begin{align*}
		\P^{\otimes N}(U_{k,l,\tau})&=	\int\ldots \int
		\mathbf 1_{V_{\tau}}(\sigma(\underline v))
		\prod _{i=1}^N\bigg(\prod_{j=1}^r p(x_1^{i,j},\d\zeta^{i,j})p(z^{i,j},\d u^{i,j})\bigg)p(x_1^{i,r+1},\d\zeta^{i, r+1}).
	\end{align*}
	In this expression, there are $N(2r+1)$ integrals that are taken over $\Wc_{k_{1,1}}$, $\Wc_{l_{1,1}}$,$\ldots$, $\Wc_{l_{N,r}}$, $\Wc_{k_{N,r+1}}$. Since the words $\zeta^{i,j}$ are not constrained by the indicator function, the integrals over $\Wc_{k_{i,j}}$ simplify; by definition of $p^k$ we get 
	\begin{align*}
		\P^{\otimes N}(U_{k,l,\tau})=
		\int\ldots \int
		\mathbf 1_{V_{\tau}}(\sigma(\underline v))
		\prod_{i=1}^N\prod _{j=1}^rp^{k_{i,j}}(x_1^{i,j},\d u ^{i,j}),
		\end{align*}
		where the $Nr$ remaining integrals are taken over $\Wc_{l_{1,1}},\Wc_{l_{1,2}},\ldots ,\Wc_{l_{N,r}}$. Note that, during this simplification, the integrals over $\Wc_{k_{i,r+1}}$ became $p(x_1^{i,r+1},\Wc_{k_{i, r+1}})=1$.
		We are now ready to use Assumption~\ref{hyp: pseudo uniformity with tau} in each of the remaining integrals. 
		We recall that $x_2^{i,j}$ is a letter depending on $\underline u$ that was defined in the beginning of the proof. 
		By definition, for all $1\leq j\leq r$, the first letter of $u^{i,j}$ belongs to $K_j$ if $u^{i,j}\neq e$ and $x_1^{i,j}$ and $x^{i,j}_2$ both belong to $K_j^-$. 
		By Assumption~\ref{hyp: pseudo uniformity with tau}, we have 
	\begin{align}
		\label{eqloc: PN(Uklt) before developping ptauij}
		\P^{\otimes N}(U_{k,l,\tau})
		&\leq c_K^{Nr}\int\ldots \int
		\mathbf 1_{V_\tau} (\sigma (\underline v))
		\prod_{i=1}^N\prod _{j=1}^rp^{\tau_{i,j}}(x_2^{i,j},\d u ^{i,j}),
	\end{align}
	where the integrals are still taken over $\Wc_{l_{1,1}},\Wc_{l_{1,2}},\ldots ,\Wc_{l_{N,r}}$.
	Before expanding back each $p^{\tau_{i,j}}$ into an integral over $\Wc_{\tau_{i,j}}$, we need one last notation. Given words $\xi^{i,j}$ such that $|\xi^{i,j}|=\tau_{i,j}$ and another word $\xi^{1,r+1}$ such that 
	\begin{equation*}
		|\xi^{1,r+1}|=\tau_{1,r+1}:=T-(\tau_{1,1}+l_{1,1}+\ldots +\tau_{N,r}+l_{N,r}),
	\end{equation*} we define $y^{i,j}$ as the letter preceding $u^{i,j}$ in the word
	\begin{equation}
		\label{eqloc: change of variable long word}
		w:=(\xi^{1,1}u^{1,1}\xi^{2,1}u^{2,1}\ldots \xi^{N,1}u^{N,1})\ldots (\xi^{1,r}u^{1,r}\xi^{2,r}u^{2,r}\ldots \xi^{N,r}u^{N,r})\xi^{1,r+1},
	\end{equation}
	or $\xinit$ if there are none~---~of course, the letter $y^{i,j}$ depends implicitly on the choice of words $\xi^{i,j}$.
	With this notation, we have
	\begin{equation*}
		 p^{\tau_{i,j}}(x_2^{i,j},\d u ^{i,j})=\int_{W_{\tau_{i,j}}}p(x_2^{i,j},\d \xi^{i,j})p(y^{i,j},\d u^{i,j}).
	\end{equation*}
	Injecting these expansions in~\eqref{eqloc: PN(Uklt) before developping ptauij}, we get
	\begin{align}
	\label{eqloc: PN(Uklt) after developing tauij}
	\P^{\otimes N}(U_{k,l,\tau})
	&\leq c_K^{Nr}\int\ldots \int
	\mathbf 1_{V_\tau} (\sigma (\underline v))
	\prod_{i=1}^N\bigg(\prod _{j=1}^rp(x_2^{i,j},\d \xi ^{i,j})p(y^{i,j},\d u^{i,j})\bigg),
	\end{align}
	where there are now $2Nr$ integrals over $\Wc_{\tau_{1,1}}$, $\Wc_{l_{1,1}}$, $\Wc_{\tau_{2,1}}$, $\Wc_{l_{2,1}}$, $\ldots$, $\Wc_{\tau_{N,r}}$, $\Wc_{l_{N,r}}$. Using the fact that all $x\in\Rxinit $ satisfy
	\begin{equation*}
		1=p(x, \Wc_{\tau_{1,r+1}})=\int_{\Wc_{\tau_{1,r+1}}}p(x,\d \xi ^{1,r+1}),
	\end{equation*}
	we can even insert one last integral over $\xi^{1,r+1}\in \Wc_{\tau_{1,r+1}}$ in this bound. Moreover, by definition of $V_\tau$, the indicator function in the integrals satisfies $\mathbf 1_{V_{\tau}} (\sigma (\underline v))\leq \mathbf 1_W(w)$, where $w$ is the word defined in~\eqref{eqloc: change of variable long word}. By performing the change of variable induced by~\eqref{eqloc: change of variable long word}, we get 
	\begin{equation}
		\label{eqloc: crude bound on P(Uklt)}
		\P^{\otimes N}(U_{k,l,\tau})
		\leq c_K^{Nr}\int_{\Wc_T}\mathbf 1_{W}(w)p(\xinit,\d w)
		= c_K^{Nr}\P(W).
	\end{equation}
	This bound holds for every choice of $k, l, \tau$. Since the sets $U_{k,l,\tau}$ provide a partition of $U$, we just need to count the number of different choices of $k,l,\tau$ to obtain a bound on $\P^{\otimes N}(U)$. There are at most $ (n+1)^{N(2r+1)}$
	 joint choices of $k$ and $l$, and at most $(\tau_K+1)^{Nr}$ choices for $\tau$. By multiplying these factors, we recover the constant of~\eqref{eq: proba ineq coupling map}. The proof is complete.

\end{proof}
\subsection{The Ruelle-Lanford function for admissible measures}
\label{section: coupling and decoupling}
In this section, we prove Proposition~\ref{prop: IRL for admissible measures}. Assume~\ref{hyp: pseudo lsc density},~\ref{hyp: pseudo uniformity} and~\ref{hyp: null border}.
Let $\mu_1,\mu_2$ be two admissible measures such that the probability measure $\mu:=\frac12\mu_1+\frac12\mu_2$ is admissible too.
We will take $\mu_1=\mu_2=\mu$ to prove the existence of the Ruelle-Lanford function. We will then choose $\mu_1$ and  $\mu_2$ arbitrarily to prove the convexity property of Proposition~\ref{prop: IRL for admissible measures}.
Since $\mu$ is absolutely continuous with respect to $\leb$,  
Assumption~\ref{hyp: null border} yields that $\mu(\partial C)=0$ so that $\mu(C)=1$. 
Let $\epsilon >0$.
We choose a finite subfamily of pairwise comparable classes of $(C_j)_{j\in \mathcal J}$, which we now relabel as $C_1, C_2,\ldots ,C_r$, such that $\mu(C_j)>0$ for all $j\leq r$ and
\begin{align*}
	\mu(C_1)+\ldots +\mu(C_r)\geq 
	1-\frac\epsilon4
	,\qquad	 \beta\leadsto1\leadsto 2\leadsto \ldots \leadsto r.
\end{align*}
The existence of such a subfamily of $(C_j)_{j\in \mathcal J}$ is granted by the admissibility of $\mu$.
Since $\mu=\frac12\mu_1+\frac12\mu_2$, we have 
\begin{equation*}
	\mu_k(C_1)+\ldots +\mu_k(C_r)\geq 1-\frac\epsilon2,\qquad k\in \{1,2\}.
\end{equation*}
The measures $\mu_1$ and $\mu_2$ are inner regular because they are absolutely continuous with respect to the Lebesgue measure.
Hence, there exists a compact set $K^0\subseteq C_1\cup \ldots \cup C_r$ such that $\mu_1(K ^0)\geq 1-\epsilon$ and $\mu_2(K ^0)\geq 1-\epsilon$.
As a consequence, we also have $\mu(K^0)\geq 1-\epsilon$.
Let $\delta>0$. Without loss of generality, we assume that
\begin{equation*}
	0<\delta<\min\Big(d(K^0, \R^d\setminus (C_1\cup \ldots \cup C_r)), \frac{1-\epsilon}2\Big),
\end{equation*} 
so that
\begin{equation*}
	K:= \{x\in \R^d\ |\ d(x,K^0)\leq \delta\}\subseteq C_1\cup \ldots \cup C_r.
\end{equation*} 
The set $K$ is an admissible compact set.
As before, we set $K_\beta=C_\beta=\{\xinit\}$ and $K_j=C_j\cap K$ for all $1\leq j\leq r$. The slicing, stitching, and coupling maps are henceforth defined with respect to these choices of sets $C_j$, $K_j$ and $K$.

In this context, Proposition~\ref{prop: geographic ineq coupling map} results in an $\epsilon$-$\delta$ bound, for proper choices of $n$ and $T$: this is Corollary~\ref{coro: advanced geographic ineq coupling map} below.
Corollary~\ref{coro: advanced geographic ineq coupling map} states that the coupling map turns any list of words $\underline u$ whose empirical measures approximate $\mu_1$ and $\mu_2$ into a set of words $w$ whose empirical measure approximates $\mu$.
We recall that $h$ is simply the continuous function satisfying $h(1)=0$ introduced in~\eqref{eq: function h}. Notice that the right-hand side of~\eqref{eq: fine geographic ineq coupling map} vanishes as $(\epsilon, \delta)\to (0,0)$.
\begin{coro}
	\label{coro: advanced geographic ineq coupling map}
	Let $n\in \N$ be such that $\tau_{K}/n\leq \delta$ and $n(1-\epsilon-\delta)\geq 1$, let $T\in \N$ be such that $n+r\tau_{K}\leq T\delta$, and let\footnote{This choice of $N$ satisfies $T\geq Nn+Nr\tau_K$.}
	\begin{equation*}
		N=\left\lfloor\frac T{n+r\tau_K}\right\rfloor.
	\end{equation*}
	Let
	$\underline u=(u^1,\ldots ,u^N)\in \W_n^N$ be such that $\dlp (L[u^i], \mu_1)\leq \delta$ when $i$ is odd and $\dlp (L[u^i], \mu_2)\leq \delta$ when $i$ is even. Then,
	for all $w\in \cou_{N,n,T}(\underline u)$,
	\begin{equation}
		\label{eq: fine geographic ineq coupling map}
		\dlp(L[w], \mu)\leq 4h\Big((1-\epsilon-\delta)\frac{1-\delta}{1+r\delta}\Big)+3\delta=:f(\epsilon, \delta).
	\end{equation}
\end{coro}
\begin{proof}
	By definition of $K$ and since each $L[u^i]$ is $\delta$-close to $\mu_1$ or $\mu_2$, we have $L[u^i](K)\geq 1-\epsilon-\delta$. Hence at least one letter of each $u^i$ belongs to $K$, implying that $\underline u^i:=\sli_n(u^i)\neq(e,\ldots,e)$. 
	By Proposition~\ref{prop: geographic ineq coupling map}, we have
	\begin{equation*}
		\dlp(L[w], \mu)\leq\frac 1N\sum_{i=1}^N\Big(h\Big(\frac{N|\underline u^i|}{|\underline v|}\Big)+h\Big(\frac{|\underline u^i|}{n}\Big)\Big)+ 2h\Big(\frac{|\underline v|}{T}\Big)+\dlp(L[\underline u],\mu),
	\end{equation*}
	where $\underline v$ is as in Proposition~\ref{prop: geographic ineq coupling map}.
	Since $L[\underline u]=\frac 1N L[u^1]+\ldots+\frac 1NL[u^N]$ and $\mu=\frac12\mu_1+\frac12\mu_2$, Lemma~\ref{lemma: dlp with convex sums} allows us to bound the last term by
	\begin{equation*}
		\delta+ \bigg|\frac12-\frac{\lceil\frac N2\rceil}{N}\bigg|
		+ \bigg|\frac12-\frac{\lfloor\frac N2\rfloor}{N}\bigg|\leq \delta+\frac 1N\leq 3\delta.
	\end{equation*}
	For the other terms, we need to prove that the quantities $\frac{N|\underline u^i|}{|\underline v|}$, $\frac{|\underline u^i|}{n}$ and $\frac{|\underline v|}{T}$ are close to $1$. Let $1\leq i\leq N$ and let $k=1$ if $i$ is odd and $k=2$ if $i$ is even. By the choice of $K$ as the $\delta$-neighborhood of $K^0$, the bound $\dlp(L[u^i],\mu_{k})\leq \delta$ yields 
	\begin{equation*}
		L[u^i](K)+\delta\geq \mu_k(K^0)\geq 1-\epsilon,
	\end{equation*}
	hence $u^i$ has at least $(1-\epsilon -\delta)n$ letters in $K$. Thus, by definition of $\underline u^i$, we have
	\begin{equation*}
			1-\epsilon -\delta\leq \frac{|\underline u^i|}{n}\leq 1.
	\end{equation*}
	Since $|\underline v|=|\underline u^1|+\ldots +|\underline u^N|$, we have
	\begin{align*}
		1-\epsilon -\delta\leq \frac{N|\underline u^i|}{|\underline v|}&\leq \frac 1{1-\epsilon -\delta}.
	\end{align*}
	Moreover, by definition of $N$, $T$ and $n$, we have
	\begin{equation*}
		(1-\epsilon-\delta)\frac{1-\delta}{1+r\delta}\leq \frac {|\underline v|}T\leq 1.
	\end{equation*}
	Considering the monotonicity of $h$ on $(0,1)$ and on $(1, \infty)$, as well as the fact that $h(1/x)=h(x)$, we have 
	\begin{equation*}
		\frac 1N\sum_{i=1}^N\Big(h\Big(\frac{N|\underline u^i|}{|\underline v|}\Big)+h\Big(\frac{|\underline u^i|}{n}\Big)\Big)+ 2h\Big(\frac{|\underline v|}{T}\Big) \leq 4h\Big((1-\epsilon-\delta)\frac{1-\delta}{1+r\delta}\Big).
	\end{equation*}
	The bound of the corollary follows.
\end{proof}
\begin{lem}
	\label{theo: decoupling ineq}
	Let $n,$ $T$ and $N$ be as in Corollary~\ref{coro: advanced geographic ineq coupling map}. Let 
	\begin{align*}
		A_n^{(1)}(\delta)&=\{u\in \W_n\ |\ \dlp(L[u],\mu_1)\leq \delta\},\\
		A_n^{(2)}(\delta)&=\{u\in \W_n\ |\ \dlp(L[u],\mu_2)\leq \delta\},\\
		B_T(\epsilon, \delta)&=\{w\in \Wc_T\ |\ \dlp(L[w],\mu)\leq f(\epsilon, \delta)\}.
	\end{align*}
	Then,
	\begin{equation}
		\begin{split}
		\label{eq: supermultiplicative inequality}
		\P(A_n^{(1)}(\delta))^{\lceil N/2\rceil}\P(A_n^{(2)}(\delta))^{\lfloor N/2\rfloor}&\leq \ccou \P(B_T(\epsilon,\delta)),
		\\ \ccou&= \big(c_K^r(\tau_K+1)^{r}(n+1)^{2r+1}\big)^{N}.
		\end{split}
	\end{equation}
\end{lem}
\begin{proof}
	By Corollary~\ref{coro: advanced geographic ineq coupling map}, 
	\begin{align*}
		\cou_{N,n,T}(A_{n,N}(\delta))&\subseteq B_T(\epsilon, \delta),
		\\\hbox{where } A_{n,N}(\delta)&:=A_n^{(1)}(\delta)\times A_n^{(2)}(\delta)\times A_n^{(1)}(\delta)\times A_n^{(2)}(\delta)\times\ldots \times A_n^{(2-N+2\lfloor N/2\rfloor)}(\delta).
	\end{align*}
	Hence, by Proposition~\ref{prop: proba ineq coupling map}, we have
	\begin{align*}
		\P^{\otimes N}(A_{n,N}(\delta))
		&\leq \P^{\otimes N }(\{\underline u\in \W_n^N\ |\ \cou_{N,n,T}(\underline u)\subseteq\cou_{N,n,T}(A_{n,N}(\delta))\})
		\\&\leq \big(c_K^r(\tau_K+1)^{r}(n+1)^{2r+1}\big)^{N} \P(\cou_{N,n,T}(A_{n,N}(\delta)))
		\\&\leq \big(c_K^r(\tau_K+1)^{r}(n+1)^{2r+1}\big)^{N} \P (B_{T}(\epsilon,\delta)).
	\end{align*}
	The proof is complete.
\end{proof}
Lemma~\ref{theo: decoupling ineq} provides the supermultiplicative inequality\footnote{\eqref{eq: supermultiplicative inequality} is not \emph{stricto sensu} a supermultiplicative inequality, yet it is enough for the purpose of proving the existence of the Ruelle-Lanford function.} mentioned in the discussion below Lemma~\ref{lem: RL function implies LDP}. We now use this inequality to prove Proposition~\ref{prop: IRL for admissible measures}, effectively completing the proof of the weak LDP.
\begin{proof}[Proof of Proposition~\ref{prop: IRL for admissible measures}]
	\addcontentsline{toc}{subsubsection}{Proof of Proposition~\ref{prop: IRL for admissible measures}}
	We shall divide the logarithm of~\eqref{eq: supermultiplicative inequality} by $T$ and pass to the limit as $n\to \infty$ and $T\to\infty$ to prove the existence of the Ruelle-Lanford function.
	Before proceeding, notice that using $N=\big\lfloor \frac{T}{n+r\tau_K}\big\rfloor$ in the definition of $\ccou$ yields
	\begin{equation*}
		\lim_{T\to \infty}\frac 1T\log\ccou=\frac{1}{n+r\tau_K}\log\big(c_K^r(\tau_K+1)^{r}(n+1)^{2r+1}\big),
	\end{equation*}
	thus
	\begin{equation}
		\label{eqloc: limit of ccou}
		\lim_{n\to \infty}\lim_{T\to \infty}\frac 1T\log \ccou=0.
	\end{equation}
	We can now proceed to the proof of the existence of the Ruelle-Lanford function.
	Let $\mu_1=\mu_2=\mu$ in Lemma~\ref{theo: decoupling ineq} and let $A_n(\delta):=A_{n}^{(1)}(\delta)=A_n^{(2)}(\delta)$. For suitable integers $n$ and $T$, this lemma yields
	\begin{equation*}
		\frac 1T\left\lfloor\frac{T}{n+r\tau_K}\right\rfloor\log\P(A_n(\delta))\leq \frac 1T\log \ccou +\frac 1T\log \P(B_T(\epsilon,\delta)).
	\end{equation*} 
	In this inequality, we take first the limit inferior as $T\to \infty$, and then the limit superior as $n\to \infty$. By~\eqref{eqloc: limit of ccou}, we have
	\begin{equation}
		\label{eqloc: ssup(B) leq sinf(B)}
		\overline s(\Bc_{\LP}(\mu, \delta))\leq 0+ \underline s(\Bc_{\LP}(\mu, f(\epsilon, \delta))).
	\end{equation} 
	Since $\lim_{\epsilon \to 0}\lim_{\delta\to 0} f(\epsilon, \delta)=0$,
	taking the limit in~\eqref{eqloc: ssup(B) leq sinf(B)} as $\delta\to 0$, and then as $\epsilon \to 0$ yields $\overline s(\mu)\leq \underline s(\mu)$, which  shows that the Ruelle-Lanford function exists at $\mu$.
	We now prove that the rate function $\IRL$ satisfies~\eqref{eq: convexity of IRL}. For any choice of admissible $\mu_1$ and $ \mu_2$ such that $\mu:=\frac12\mu_1+\frac12\mu_2$ is admissible, and for proper choice of $n$ and $T$, Lemma~\ref{theo: decoupling ineq} yields
	\begin{equation*}
			\frac {\big\lceil \frac N2\big \rceil}T\log\P(A_n^{(1)}(\delta))+\frac {\big\lfloor \frac N2\big \rfloor}T\log\P(A_n^{(2)}(\delta))\leq \frac 1T\log \ccou +\frac 1T\log \P(B_T(\epsilon,\delta)).
	\end{equation*}
	In this inequality, we take first the limit inferior as $T\to \infty$, and then the limit inferior as $n\to \infty$. By~\eqref{eqloc: limit of ccou}, we have
	\begin{equation*}
			\frac 12\underline s(\Bc_{\LP}(\mu_1, \delta))+\frac 12\underline s(\Bc_{\LP}(\mu_2, \delta))\leq 0+ \underline s(\Bc_{\LP}(\mu, f(\epsilon, \delta))).
	\end{equation*}
	Since the Ruelle-Lanford function exists at $\mu_1$, $\mu_2$ and $\mu$,
	taking this bound to the limit as $\delta\to 0$, and then as $\epsilon \to 0$ yields
	\begin{equation}
		\label{eqloc: midpoint convexity IRL}
		-\frac 12\IRL(\mu_1)-\frac 12\IRL(\mu_2)\leq -\IRL (\mu).
	\end{equation}
	Let now $\nu_1,\nu_2\in \Ac$ and $\lambda\in (0,1)$ be such that $\nu:=\lambda\nu_1+(1-\lambda)\nu_2\in \Ac$. By the discussion following Definition~\ref{defi: admissible measure}, we have $[\nu_1,\nu_2]\subseteq \Ac$. Therefore, by repeatedly applying~\eqref{eqloc: midpoint convexity IRL} along the dyadic approximations of $\lambda$, we obtain that 	
	\begin{equation*}
		\lambda\IRL(\nu_1)+(1-\lambda)\IRL(\nu_2)\geq \liminf_{\nu'\to\nu}\IRL (\nu').
	\end{equation*}
	Since $\IRL$ is lower semicontinuous, this proves~\eqref{eq: convexity of IRL}.
\end{proof}
\subsection{An additional property of the Ruelle-Lanford function}
\label{section: decoupling}
In this section, we prove Proposition~\ref{prop: IRL linearity}. The results of this section are not used in the proof of the weak LDP itself. Assume~\ref{hyp: pseudo lsc density},~\ref{hyp: pseudo uniformity} and~\ref{hyp: null border}.
Let $\lambda_1,\lambda_2\in (0,1)$ be such that $\lambda_1+\lambda_2=1$ and let $\mu_1,\mu_2$ be two admissible measures such that the probability measure $\mu:=\lambda_1\mu_1+\lambda_2\mu_2$ is admissible too. 
We assume that $\Jc$ is partitioned into two disjoint sets $\Jc_1$ and $\Jc_2$ such that 
\begin{equation*}
	\supp \mu_1\subseteq \bigcup_{j\in \Jc_1}\overline C_j,\qquad
	\supp \mu_2\subseteq \bigcup_{j\in \Jc_2}\overline C_j.
\end{equation*}
By Proposition~\ref{prop: IRL for admissible measures}, the quantities $\IRL(\mu)$, $\IRL(\mu_1)$ and $\IRL(\mu_2)$ exist in $[0,\infty]$. We shall prove that they satisfy~\eqref{eq: linearity 2}.
Let $\epsilon>0$. Without loss of generality, we assume that
\begin{equation}
	\label{eqloc: epsilon smaller than min}
	0<\epsilon<\min\Big(\frac {\lambda_1}2,\frac{\lambda_2}2\Big)=\min\Bigg(\frac12\mu\Bigg(\bigcup_{j\in \Jc_1}C_j\Bigg),\frac12\mu\Bigg(\bigcup_{j\in \Jc_2}C_j\Bigg)\Bigg).
\end{equation}
Since $\mu_1$ and $\mu_2$ are absolutely continuous with respect to $\leb$,  
Assumption~\ref{hyp: null border} yields that $\mu(\partial C)=\mu_1(\partial C)=\mu_2(\partial C)=0$ so that $\mu(C)=\mu_1(C)=\mu_2(C)=1$.
We choose a finite subfamily of pairwise comparable classes of $(C_j)_{j\in \mathcal J}$, that we now relabel as $C_1, C_2,\ldots ,C_r$, such that $\mu(C_j)>0$ and
\begin{align*}
	\begin{cases}
		\mu_1(C_1)+\mu_1(C_2)+\ldots +\mu_1(C_r)\geq 
		1-\frac\epsilon2
		,\\
		\mu_2(C_1)+\mu_2(C_2)+\ldots +\mu_2(C_r)\geq 1-\frac\epsilon2,
		\\
		\beta\leadsto1\leadsto 2\leadsto \ldots \leadsto r.
	\end{cases}
\end{align*}
The existence of such a subfamily of $(C_j)_{j\in \mathcal J}$ is granted by the admissibility of $\mu$.
The set $C:=C_1\cup\ldots \cup C_r$ is split in two according to the partition $\Jc=\Jc_1\sqcup\Jc_2$.
We let $\eta:\{1,\ldots, r\}\to \{1,2\}$ be the map defined by $\eta(j)=1$ if $j\in \Jc_1$ and $\eta(j)=2$ if $j\in\Jc_2$, and we set
\begin{equation*}
	C^{(1)}=\bigcup_{j\in \eta^{-1}(1)}C_j\subseteq C,\qquad C^{(2)}=\bigcup_{j\in \eta^{-1}(2)}C_j\subseteq C.
\end{equation*}
Since $\lambda_1$ and $\lambda_2$ are no larger than 1, we have 
	$\mu\big(C^{(1)}\big)\geq \lambda_1-\epsilon/ 2>0$
	and  $\mu\big(C^{(2)}\big)\geq \lambda_2-\epsilon /2>0.$
Since $\mu$ is absolutely continuous with respect to the Lebesgue measure, Assumption~\ref{hyp: null border} implies that $C^{(1)}$ and $C^{(2)}$ are continuity sets of $\mu$.
By the Portmanteau theorem, there exists $\delta>0$ such that $|\nu(C^{(1)})-\mu(C^{(1)})|< \epsilon/ 2$ and $|\nu(C^{(2)})-\mu(C^{(2)})|< \epsilon/ 2$ for all $\nu\in\Bc_\LP(\mu,\delta)$, 
further implying 
\begin{equation}
	\label{eq: nu (Cgamma)}
	\nu\big(C^{(1)}\big)> \lambda_1-\epsilon>0,\qquad \nu\big(C^{(2)}\big)>\lambda_2-\epsilon>0,\qquad \nu\in\Bc_\LP(\mu,\delta).
\end{equation}
The measures $\mu_1$ and $\mu_2$ are inner regular because they are absolutely continuous with respect to the Lebesgue measure.
Hence, there exists a compact set $K^0\subseteq C$ such that $\mu_1(K ^0)\geq 1-\epsilon$ and $\mu_2(K ^0)\geq 1-\epsilon$.
As a consequence, we also have $\mu(K^0)\geq 1-\epsilon$.
Since $K$ does not depend on $\delta$, let us assume that
\begin{equation*}
	0<2\delta<\min\big(d(K^0, \R^d\setminus C), {\lambda_1-\epsilon},\lambda_2-\epsilon\big),
\end{equation*}
without loss of generality.
We set
\begin{equation*}
	K= \{x\in \R^d\ |\ d(x,K^0)\leq \delta\}\subseteq C.
\end{equation*} 
The set $K$ is an admissible compact set.
As before, we set $K_\beta=C_\beta=\{\xinit\}$ and $K_j=C_j\cap K$ for all $1\leq j\leq r$. By the assumption on $\delta$, note that $d(K_j,\partial C_j)>\delta$ for all $1\leq j\leq r$.
The set $K$ is also split in two according to the partition $\Jc=\Jc_1\sqcup\Jc_2$; we set
\begin{equation*}
	K^{(1)}=\bigcup_{j\in \eta^{-1}(1)}K_j,\qquad K^{(2)}=\bigcup_{j\in \eta^{-1}(2)}K_j.
\end{equation*}

\subsubsection{The decoupling map}
\label{section: decoupling map}
We reuse the slicing and stitching maps of Section~\ref{section: slicing and stitching} and compose them in an alternative way to define the \emph{decoupling map}.
For $\gamma\in \{1,2\}$, we set $r_\gamma=\#\eta^{-1}(\gamma)$ and we denote by $j_\gamma(i)$ the $i$-th integer $j$ satisfying $\eta(j)=\gamma$, so that 
\begin{equation*}
	\{1,\ldots,r\}=\{j_1(1),\ldots,j_{1}(r_1)\}\sqcup \{j_2(1),\ldots,j_2(r_2)\}.
\end{equation*}
For all $u\in \Wc$, we denote by $|u|_\gamma=|u|_{C^{(\gamma)}}$ the number of letters of $u$ that belong to $C^{(\gamma)}$, and 
\begin{equation*}
	L^{(\gamma)}[u]=\frac 1{|u|_{\gamma}}\sum_{i=1}^{|u|}\mathbf 1_{C^{(\gamma)}}(u_i)\delta_{u_i}\in \Pc(\R^d),
\end{equation*}
when $|u|_\gamma>0$.
We fix two integers $T_1,T_2$ such that $T_{\gamma}\geq (\lambda_\gamma+\epsilon) n+r_\gamma\tau_K$ for $\gamma\in \{1,2\}$.
\begin{defi}[The decoupling map]
	Let $n\in \N$. Let $u\in \W_n$ be such that $|u|_\gamma\leq (\lambda_\gamma+\epsilon)n$ for $\gamma\in \{1,2\}$ and set $(u^1,\ldots ,u^r)=\sli_n(u)$.
	For $\gamma\in \{1,2\}$, we denote by $\underline v^{(\gamma)}$ the list $ (u^{j_\gamma(1)},\ldots, u^{j_\gamma(r_\gamma)})$.
	Then $(\underline v^{(1)},\underline v^{(2)})\in \Sc_{r_1, T_1}\times \Sc_{r_2,T_2}$ and we set 
	\begin{equation*}
		\decou_n^{(1)}(u)=\sti_{r_1, T_1}\big(\underline v^{(1)}\big), \qquad		\decou_n^{(2)}(u)=\sti_{r_2,T_2 }\big(\underline v^{(2)}\big), \qquad\decou_n(u)=\big(\decou_n^{(1)}(u),\decou_n^{(2)}(u)\big).
	\end{equation*}
\end{defi}
In this definition, we have indeed $(\underline v^{(1)},\underline v^{(2)})\in \Sc_{r_1, T_1}\times \Sc_{r_2,T_2}$ because the two lists $\underline v^{(\gamma)}$ are properly ordered and satisfy $|\underline v^{(\gamma)}|\leq |u|_\gamma\leq (\lambda_\gamma+\epsilon )n$, for $\gamma\in\{1,2\}$.
\begin{figure}[!tb]
	\centering
	\begin{subfigure}[b]{1\textwidth}
		\includegraphics[width=\textwidth]{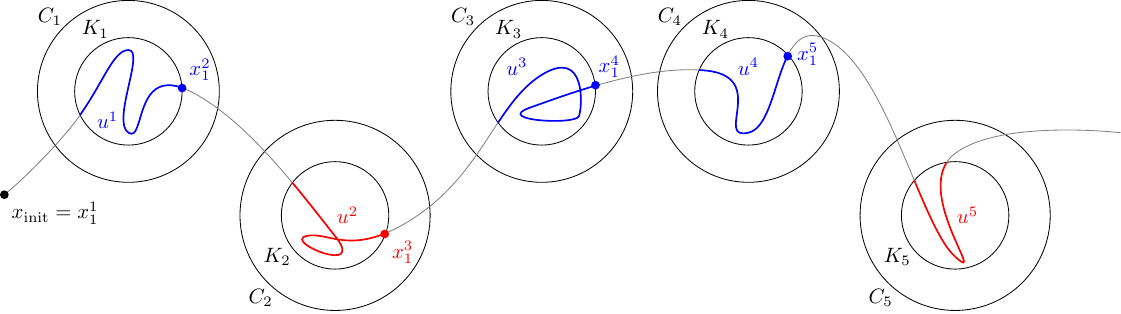}
		\captionsetup{width=0.7\textwidth}
		\captionsetup[subfigure]{skip=1em}
		\caption{The action of the slicing map on a word $u$.
			The letters $x_1^{i,j}$ used in the proof of Proposition~\ref{prop: proba ineq decoupling map} are indicated.\vspace{0.5cm}}
		\label{fig: slicing map for decoupling}
	\end{subfigure}
	\begin{subfigure}[b]{1\textwidth}
		\includegraphics[width=\textwidth]{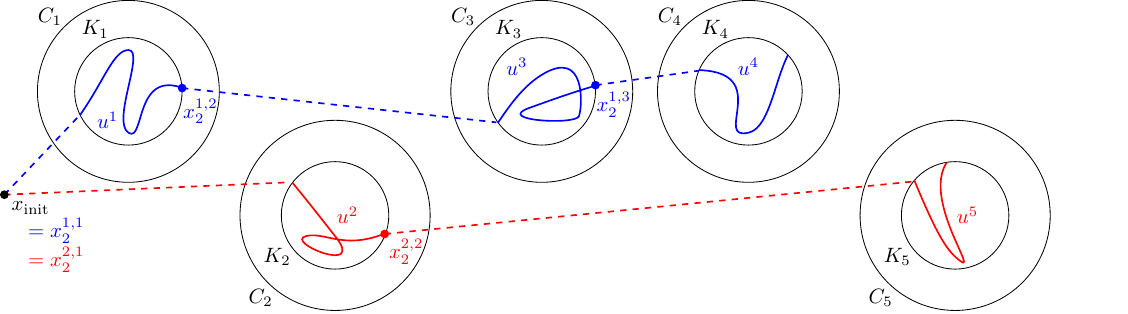}
		\captionsetup{width=0.7\textwidth}
		\caption{The action of the stitching map on the stitchable lists $(u^1,u^3,u^4)$ and $(u^2,u^5)$. The dashed lines represent the insertion of $\Wc_{\tau_k}$ between words. The letters $x_2^{\gamma,i}$ used in the proof of Proposition~\ref{prop: proba ineq decoupling map} are indicated.}
		\label{fig: stitching map for decoupling}
	\end{subfigure}
	\captionsetup{width=0.8\textwidth}
	\caption{Example of use of the decoupling map with respect to the partition $\{1,2,3,4,5\}=\{1,3,4\}\sqcup\{2,5\}$. Trajectories are depicted as smooth curves for easier readability, even though they are discrete in time. Given a word $u$, the slicing map produces five subwords $u^1,\ldots u^5$ in~\eqref{fig: slicing map for decoupling}. These subwords are then reassembled by the stitching map in~\eqref{fig: stitching map for decoupling}. }
	\label{fig: decoupling map}
\end{figure}

\begin{prop}
	\label{prop: geographic ineq decoupling map}
	Let $u\in \W_n$ be such that $|u|_\gamma\leq (\lambda_\gamma+\epsilon) n$ for $\gamma\in \{1,2\}$ and set $(u^1,\ldots, u^r)=\sli_{n}(u)$. For $\gamma\in \{1,2\}$, we denote by $\underline v^{(\gamma)}$ the list $ (u^{j_\gamma(1)},\ldots, u^{j_\gamma(r_\gamma)})$. 
	Assume that $\underline v^{(\gamma)}\neq(e,\ldots,e)$ for $\gamma\in \{1,2\}$.
	Then, for all $w\in \decou_n(u)$, 
	\begin{equation}
		\label{eq: geographic ineq decoupling map}
		\dlp(L^{(\gamma)}[u],L[w])\leq  h\Big(\frac {{|\underline v^{(\gamma)}|}}{|u|_\gamma}\Big)+2h\Big(\frac{|\underline v^{(\gamma)}|}{T_\gamma}\Big).
	\end{equation}
\end{prop}
\begin{proof}
	Since $\dlp \leq \dtv$, we have
	\begin{equation}
		\label{eqloc: bound dlp Liu Lw}
		\dlp(L^{(\gamma)}[u],L[w])\leq\dtv\big(L^{(\gamma)}[u],L[\underline v^{(\gamma)}]\big)+\dtv\big(L[\underline v^{(\gamma)}],L[w]\big).
	\end{equation}
	The first term of this bound yields the first term on the right-hand side of~\eqref{eq: geographic ineq decoupling map}. Indeed, as in the proof of Proposition~\ref{prop: geographic ineq slicing map}, we have
	\begin{align*}
		\dtv\big(L^{(\gamma)}[u],L[\underline v^{(\gamma)}]\big)
		\leq
		\frac 1{{|\underline v^{(\gamma)}|}}-\frac 1{|u|_ \gamma}+\frac 1{|u|_\gamma}\bigg |\sum_{k=1}^{n}\mathbf 1_{C^{(\gamma)}}(u_k)\delta_{u_k}-\sum_{j\in \eta^{-1}(\gamma)}\sum_{k=1}^{|u^j|}\delta_{u^j_k}\bigg|_\TV.
	\end{align*}
	On the right-hand side,
	since all the letters of $u^j$ belong to $C_j$, every Dirac term that appears in the rightmost sum also appears in the left one. Hence,
	\begin{align*}
		\dtv\big(L^{(\gamma)}[u],L[\underline v^{(\gamma)}]\big)\leq 
		\frac 1{{|\underline v^{(\gamma)}|}}-\frac 1{|u|_ \gamma}+1-\frac {{|\underline v^{(\gamma)}|}}{|u|_\gamma}
		\leq h\Big(\frac {{|\underline v^{(\gamma)}|}}{|u|_\gamma}\Big).
	\end{align*}
	The second term of the bound~\eqref{eqloc: bound dlp Liu Lw}  yields the second term on the right-hand side of~\eqref{eq: geographic ineq decoupling map} by Proposition~\ref{prop: geographic ineq stitching map}.
\end{proof}
\begin{prop}
	\label{prop: proba ineq decoupling map}
	Let $W$ be a measurable set of $\Wc_{T_1}\times \Wc_{T_2}$ and let 
	\begin{equation*}
		U=\{u\in \W_n\ |\ \decou_{n,T_1,T_2}(u)\subseteq W\}.
	\end{equation*}
	Then,
	\begin{equation}
		\label{eq: proba ineq decoupling map}
		\P (U)\leq (n+1)^{2r+1}(\tau_K+1)^rc_K^r\P^{\otimes 2}(W).
	\end{equation}
\end{prop}
\begin{proof}
	Since the Borel $\sigma$-algebra of $\Wc_{T_1}\times \Wc_{T_2}$ is generated by the $\pi$-system of Cartesian products of the form $W=W_1\times W_2$ where $W_\gamma$ is a measurable set of $\Wc_{T_\gamma}$, $\gamma\in \{1,2\}$, we make the assumption that $W=W_1\times W_2$ without loss of generality.
	The general structure of this proof is the same as that of the proof of Proposition~\ref{prop: proba ineq coupling map}.
	The notation of this proof mirrors that of the proof of Proposition~\ref{prop: proba ineq coupling map}; we manipulate a word $u\in U$ on which several variables depend. We set $(u^1,\ldots ,u^r):=\sli_n(u)$ and we define $\zeta^j$ by the expression
	\begin{equation}
		\label{eqloc: decomposition u}
		u=\zeta^1u^1\zeta^2u^2\ldots \zeta^ru^r\zeta^{r+1},
	\end{equation}
	where by convention $\zeta^j=e$ when $u^j=e$. The words $\zeta^j$ implicitly depend on $u$. For $\gamma\in \{1,2\}$, we set $\underline v^{(\gamma)}$ to be the list of those $u^j$ such that $ \eta(j)=\gamma$. As in the proof of Proposition~\ref{prop: proba ineq coupling map}, we also introduce some specific letters that implicitly depend on $u$. We denote by $x_1^j$ the letter preceding $\zeta^j$ in $u$, or $\xinit$ if there are none, and we denote by $z^j$ the letter preceding $u^j$ in $u$, or $\xinit$ if there are none. In addition, for $\gamma\in\{1,2\}$ and $1\leq i\leq r_\gamma$, the word $u^{j_\gamma(i)}$ is an element of the list $\underline v^{(\gamma)}$, so we denote by $x_2^{\gamma, i}$ the last letter of the non-empty word preceding $u^{j_\gamma(i)}$ in $\underline v^{(\gamma)}$, or $\xinit$ if there are none. Letters $x_1^j$ and $x_2^{\gamma, i}$ are illustrated in Figure~\ref{fig: decoupling map}. To manipulate words of fixed length, we define the sets
	\begin{align*}
		\mathcal E_k&=\{u\in \W_n\ |\ \forall 1\leq j\leq r+1,\  |\zeta^j|=k_j\},
		\\\mathcal F_l&=\{u\in \W_n\ |\ \forall 1\leq j\leq r,\  |u^j|=l_j\},
		\\\mathcal G_{\tau_1}&=\{u\in \W_n\ |\ \forall 1\leq i\leq r_1,\  \tau(x_2^{1,i},u^{j_1(i)}_1)=\tau_{1,i}\},
		\\\mathcal G_{\tau_2}&=\{u\in \W_n\ |\ \forall 1\leq i\leq r_2,\  \tau(x_2^{2,i},u^{j_2(i)}_1)=\tau_{2,i}\},
	\end{align*}
	where $k=(k_j)$, $l=(l_j)$, $\tau_1=(\tau_{1,i})$ and $\tau_2=(\tau_{2, i})$ are sequences of integers such that $k_1+l_1+\ldots+ k_r+l_r+k_{r+1}=n$ and $\tau_{\gamma,i}\leq \tau_K$ for all $\gamma\in \{1,2\}$ and $1\leq i\leq r_\gamma$.
	We set, for $\gamma\in \{1,2\}$,
	\begin{align*}
		U_{k,l,\tau}&=U\cap \mathcal E_k\cap \mathcal F_l\cap \mathcal G_{\tau_1}\cap \mathcal G_{\tau_2},
		\qquad
		V_\tau^{(\gamma)}=V^{(\gamma)}\cap \sli_n(\mathcal G_{\tau_\gamma}),
	\end{align*}
	where 
	$V^{(\gamma)}=\{\underline v^{(\gamma)}\in \Sc_{r_\gamma,T_\gamma}\ |\ \sti_{r_\gamma,T_\gamma}(\underline v^{(\gamma)})\subseteq W_\gamma\}$.
	We must estimate $\P (U_{k,l,\tau})$.
	Performing the change of variables induced by~\eqref{eqloc: decomposition u}, we have
	\begin{align*}
		\mathbb P(U_{k,l,\tau})
		&=\int_{\mathcal E_k\cap \mathcal F_l}\mathbf 1_{U\cap \mathcal G_\tau}(u)p(\xinit, \d u)
		\\&=\int\ldots \int \mathbf 1_{V^{(1)}}(\underline v^{(1)})\mathbf 1_{V^{(2)}}(\underline v^{(2)})\mathbf 1_{\mathcal G_\tau}(u)
		\\&\hspace{3cm}\times \bigg(\prod_{j=1}^rp(x_1^j, \d\zeta^j)p(z^j,\d u^j)\bigg)p(x_1^{r+1},\d\zeta^{r+1}),
	\end{align*}
	where there are $2r+1$ integrals over $\Wc_{k_1},\Wc_{l_1}, \ldots ,\Wc_{l_r},\Wc_{k_{r+1}}$ in the second line. The integrals over $\Wc_{k_i}$ can be simplified by definition of $p^{k_i}$. We get
	\begin{equation}
		\label{eqloc: last bound before bifurcation}
		\mathbb P({U_{k,l,\tau}})=\int\ldots \int \mathbf 1_{V^{(1)}}(\underline v^{(1)})\mathbf 1_{V^{(2)}}(\underline v^{(2)})\mathbf 1_{\mathcal G_\tau}(u)\prod_{j=1}^rp^{k_j}(x_{1}^j,\d u^j),
	\end{equation}
	where only the $r$ integrals over $\Wc_{l_1},\ldots ,\Wc_{l_r}$ remain.
	Before going further, we use that $\{1,\ldots,r\}$ is partitioned into $\{j_1(1),\ldots, j_1(r_1)\}$ and $\{j_2(1), \ldots, j_2(r_2)\}$, so that the product involved in~\eqref{eqloc: last bound before bifurcation} can be split into
	\begin{equation*}
		\prod_{j=1}^rp^{k_j}(x_{1}^j,\d u^j)
		=\bigg(\prod_{i=1}^{r_1}p^{k_{j_1(i)}}\big(x_{1}^{j_1(i)},\d u^{j_1(i)}\big)\bigg)
		\bigg(\prod_{i=1}^{r_2}p^{k_{j_2(i)}}\big(x_{1}^{j_2(i)},\d u^{j_2(i)}\big)\bigg).
	\end{equation*}
	Now we can use Assumption~\ref{hyp: pseudo uniformity with tau} for each $p^{k_{j_\gamma(i)}}$ in~\eqref{eqloc: last bound before bifurcation}. By definition, the first letter of $u^{j_\gamma(i)}$ belongs to $K_{j_\gamma(i)}$ and both $x_1^{j_\gamma(i)}$ and $x_2^{\gamma, i}$ belong to $K_{j_{\gamma}(i)}^-$, for all $i\leq r_\gamma$.
	By Assumption~\ref{hyp: pseudo uniformity}, we obtain
	\begin{equation*}
		\begin{split}
			\mathbb P({U_{k,l,\tau}})\leq c_K^r\int\ldots \int 
			\bigg(\mathbf 1_{V^{(1)}}(\underline v^{(1)})\prod_{i=1}^{r_1}&p^{\tau_{1,i}}\big(x_2^{1,i}, \d u^{j_1(i)}\big)\bigg) 
			\\ &\times\bigg(\mathbf 1_{V^{(2)}}(\underline v^{(2)})\prod_{i=1}^{r_2}p^{\tau_{2,i}}\big(x_2^{2,i}, \d u^{j_2(i)}\big)\bigg),
		\end{split}
	\end{equation*}
	where the $r$ integrals are still taken over $\Wc_{l_1},\ldots, \Wc_{l_r}$. Thanks to the Fubini-Tonelli theorem and the partition of $\{1,\ldots, r\}$, we can also consider that the $r$ integrals are taken over $\Wc_{l_{j_1(1)}},\ldots ,\Wc_{l_{j_1(r_1)}}$ and $\Wc_{l_{j_2(1)}},\ldots ,\Wc_{l_{j_2(r_2)}}$.
	Notice that the variables involved in the first product do not depend on those involved in the second one and vice versa.
	Indeed, by definition, each $x^{1,i}_2$ is a letter of a word of $\underline v^{(1)}$ and each $x^{2,i}_2$ is a letter of a word of $\underline v^{(2)}$. Therefore, the multiple integral can be split into a product of two factors $\mathcal I_1$ and $\mathcal I_2$, where each factor $\mathcal I_\gamma$ consists in the $r_\gamma$ integrals taken over $\Wc_{l_{j_\gamma(1)}},\ldots ,\Wc_{l_{j_\gamma(r_\gamma)}}$:
	\begin{equation}
		\begin{split}
			\label{eqloc: Pukltau bound by a product}
			\mathbb P(U_{k,l,\tau})\leq 
			c_K^r
			\mathcal I_1\mathcal I_2,
			\qquad\hbox{where}\ 
			\mathcal I_\gamma=\int\ldots \int \mathbf 1_{V^{(\gamma)}}(\underline v^{(\gamma)})\prod_{i=1}^{r_\gamma}p^{\tau_{\gamma,i}}\big(x_2^{\gamma,i}, \d u^{j_\gamma(i)}\big).
		\end{split}
	\end{equation}
	In both factors $\mathcal I_\gamma$, we want to expand each $p^{\tau_{\gamma,i}}$ into an integral. For convenience, we set $\tau_{\gamma,r_\gamma+1}=T_\gamma-(\tau_{\gamma,1}+\ldots+\tau_{\gamma,r_\gamma})-(l_{j_\gamma(1)}+\ldots+l_{j_\gamma(r_\gamma)})$. 
	Given words $\xi^{\gamma, 1},\ldots,\xi^{\gamma,r_\gamma+1}$ of respective lengths $\tau_{\gamma, 1}, \ldots, \tau_{\gamma, r_\gamma+1}$, we introduce the word
	\begin{equation}
		\label{eqloc: wgamma}
		w^{\gamma}=\xi^{\gamma,1}u^{j_\gamma(1)}\xi^{\gamma, 2}u^{j_{\gamma}(2)}\ldots \xi^{\gamma,r_\gamma}u^{j_\gamma(r_\gamma)}\xi^{\gamma, r_\gamma+1}.
	\end{equation}
	Let $y^{\gamma, i}$ denote the letter preceding $u^{j_\gamma(i)}$ in $w^\gamma$, or $\xinit$ if there are none. The letter $y^{\gamma,i}$ depends implicitly on $u$ and on the choice of $\xi^{\gamma,i}$. By definition of $p^{\tau_{\gamma,i}}$, we have
	\begin{equation*}
		\mathcal I_\gamma
		\leq \int\ldots \int\mathbf 1_{V^{(\gamma)}}(\underline v^{(\gamma)})\bigg(\prod_{i=1}^{r_\gamma}p(x_2^{\gamma, i},\d \xi^{\gamma, i}) p(y^{\gamma, i},\d u^{j_\gamma(i)})\bigg),
	\end{equation*}
	where there are $2r_\gamma$ integrals taken over $\Wc_{l_{j_\gamma(1)}},\ldots ,\Wc_{l_{j_\gamma(r_\gamma)}}$ and $\Wc_{\tau_{\gamma, 1}},\ldots ,\Wc_{\tau_{\gamma, r_\gamma}}$. Since $1=p(x,\Wc_{\tau_{\gamma, r_\gamma+1}})$ for all $x\in \Rxinit$, we can even insert a $(2r_\gamma+1)$-th integral over the variable $\xi ^{\gamma,r_\gamma+1}\in \Wc_{\tau_{\gamma, r_\gamma+1}}$ in this expression. The indicator function satisfies $\mathbf 1_{V^{(\gamma)}}(\underline v^{(\gamma)})\leq \mathbf 1_{W_\gamma}(w^\gamma)$. Performing the change of variables induced by~\eqref{eqloc: wgamma}, we get
	\begin{equation*}
		\mathcal I_\gamma
		\leq \int _{\Wc_{T_\gamma}}\mathbf 1_{W_\gamma}(w^\gamma)p(\xinit, \d w^\gamma)=\P(W_\gamma).
	\end{equation*}
	Hence,~\eqref{eqloc: Pukltau bound by a product} becomes 
	\begin{equation}
		\P(U_{k,l,\tau})\leq c_K^r\P(W_1)\P(W_2)=c_K^r\P^{\otimes 2}(W).
	\end{equation}
	It remains to estimate the number of possible choices of $k,l,\tau$. There are at most $(n+1)^{2r+1}$ joint choices for $k$ and $l$, and at most $(\tau_K+1)^r$ joint choices of $\tau_1$ and $\tau_2$. Therefore,
	\begin{equation*}
		\mathbb P(U)\leq (n+1)^{2r+1}(\tau_K+1)^rc_K^r\P^{\otimes 2}(W),
	\end{equation*}
	which concludes the proof.
\end{proof}
\subsubsection{Proof of Proposition~\ref{prop: IRL linearity}}
The bound of Proposition~\ref{prop: geographic ineq decoupling map} can be turned into an $\epsilon$-$\delta$ bound when $n$ and $T_1,T_2$ are properly chosen and $u$ is a word whose empirical measure is close to $\mu$: this is Corollary~\ref{coro: advanced geographic ineq decoupling map}.
\begin{coro}
	\label{coro: advanced geographic ineq decoupling map}
	Let $n\in \N$ be such that $n\epsilon\geq 1+r\tau_K$ and let 
	\begin{equation*}
		T_\gamma=T_\gamma(n)=\big\lceil n(\lambda_\gamma+\epsilon)\big\rceil+r_\gamma\tau_K,\qquad \gamma\in \{1,2\}.
	\end{equation*}
	Let $u\in \W_n$ be such that $\dlp(L[u],\mu)< \delta$. Then, $\decou_{n,T_1,T_2}(u)$ exists and for each $\gamma\in \{1,2\}$ and all $w^\gamma\in \decou^{(\gamma)}_{n,T_\gamma}(u)$,
	\begin{equation*}
		\dlp(L[w^\gamma],\mu_\gamma)\leq f_\gamma(\epsilon,\delta),
	\end{equation*}
	where $f_\gamma$ is a function that satisfies
	\begin{equation}
		\lim_{\epsilon\to 0}\lim_{\delta \to 0}f_\gamma(\epsilon,\delta)=0.
	\end{equation}
\end{coro}
\begin{proof}
	Let $\gamma\in \{1,2\}$.
	By~\eqref{eq: nu (Cgamma)}, since $\dlp(\mu, L[u])< \delta$, we have
	\begin{equation}
		\label{eqloc: ugamma close to nlambda}
		n(\lambda_\gamma-\epsilon)\leq |u|_\gamma\leq n(\lambda_\gamma+\epsilon).
	\end{equation}
	thus $\decou^{(\gamma)}_{n,T_\gamma}(u)$ is well defined. Let $w^\gamma\in \decou^{(\gamma)}_{n,T_\gamma}(u)$. We have 
	\begin{equation}
		\label{eqloc: triangle ineq for w gamma mu gamma}
		\dlp(L[w^\gamma],\mu_\gamma)\leq \dlp(L[w^\gamma],L^{(\gamma)}[u])+\dlp(L^{(\gamma)}[u], \mu_\gamma).
	\end{equation}
	We estimate the first term of this bound by Proposition~\ref{prop: geographic ineq decoupling map}. For the quantity on the right-hand side of~\eqref{eq: geographic ineq decoupling map} to be small, we need the quantities $|\underline v^{(\gamma)}|/|u|_\gamma$ and $|\underline v^{(\gamma)}|/T_\gamma$ to be close to $1$.
	By definition of $K$ and because $\dlp(\mu, L[u])< \delta$, we have
	\begin{equation*}
		L[u](K^{(\gamma)})\geq 
		\lambda_{\gamma}\mu_\gamma(K^0)-\delta\geq\lambda_\gamma-\epsilon -\delta.
	\end{equation*}
	Since $\underline v^{(\gamma)}$ contains all the letters of $u$ that are in $K^{(\gamma)}$ and all the letters of $\underline v ^{(\gamma)}$ belong to $C^{(\gamma)}$, we deduce 
	\begin{equation*}
		 n(\lambda_\gamma-\epsilon-\delta)\leq |\underline v^{(\gamma)}|\leq nL[u](C^{(\gamma)})=|u|_{\gamma},
	\end{equation*}
	Thus,
	\begin{equation*}
		\frac{\lambda_\gamma}{\lambda_\gamma+\epsilon}-\frac{\epsilon +\delta}{\lambda_\gamma-\epsilon}\leq \frac{|\underline v^{(\gamma)}|}{|u|_\gamma}\leq 1.
	\end{equation*}
	Moreover, since $T_\gamma\leq n(\lambda_\gamma+\epsilon)+1+r_\gamma\tau_K\leq n(\lambda_\gamma+2\epsilon)$ and $|\underline v^{(\gamma)}|\leq T_\gamma$, we have
	\begin{equation*}
 \frac{\lambda_\gamma-\epsilon -\delta}{\lambda_\gamma+2\epsilon}
 \leq
		\frac{|\underline v^{(\gamma)}|}{T_\gamma}
		\leq 1.
	\end{equation*}
	By~\eqref{eq: geographic ineq decoupling map} and the monotonicity of $h$ on $(0,1)$, we have
	\begin{equation*}
		\dlp(L[w^\gamma],L^{(\gamma)}[u])\leq h\Big(\frac{\lambda_\gamma}{\lambda_\gamma+\epsilon}-\frac{\epsilon+\delta}{\lambda_\gamma-\epsilon}\Big)+2h\Big(\frac{\lambda_\gamma-\epsilon -\delta}{\lambda_\gamma+2\epsilon}\Big),
	\end{equation*}
	which vanishes as $\delta\to 0$ and $\epsilon \to 0$.
	The estimate on the second term of the bound ~\eqref{eqloc: triangle ineq for w gamma mu gamma} comes from Lemma~\ref{lem: prop dlp on classes} applied to $\nu=L[u]$. This Lemma~\ref{lem: prop dlp on classes} yields the bound~\eqref{eq: dlp prop second} with $\nu_\gamma=L^{(\gamma)}[u]$.
Therefore,~\eqref{eqloc: triangle ineq for w gamma mu gamma} becomes
\begin{align*}
	\dlp(L[w^\gamma],\mu_\gamma)\leq
	h\Big(\frac{\lambda_\gamma}{\lambda_\gamma+\epsilon}-\frac{\epsilon +\delta}{\lambda_\gamma-\epsilon}\Big)&+2h\Big(\frac{\lambda_\gamma-\epsilon -\delta}{\lambda_\gamma+2\epsilon}\Big)
	\\&+\frac{\epsilon}{\lambda_\gamma}(1+{\lambda_{3-\gamma}})+\delta\Big(1+\frac{1}{\lambda_\gamma}\Big).
\end{align*}
Denote by $f_\gamma(\epsilon, \delta)$ the right-and side of this bound. The proof is complete.
\end{proof}
We can now use Proposition~\ref{prop: proba ineq decoupling map} and Corollary~\ref{coro: advanced geographic ineq decoupling map} to prove  Proposition~\ref{prop: IRL linearity}.
\begin{proof}[Proof of Proposition~\ref{prop: IRL linearity}]
	For $n, T_1, T_2$ defined as in Corollary~\ref{coro: advanced geographic ineq decoupling map}, we set  
		\begin{align*}
		A_n(\delta)&=\{u\in \W_n\ |\ \dlp(L[u],\mu)\leq \delta\},
		\\B^{(1)}_{T_1}(\epsilon, \delta)&=\{w ^1\in \Wc_{T_1}\ |\ \dlp(L[w^1],\mu_1)\leq f_1(\epsilon, \delta)\},
		\\B^{(2)}_{T_2}(\epsilon, \delta)&=\{w^ 2\in \Wc_{T_2}\ |\ \dlp(L[w^2],\mu_2)\leq f_2(\epsilon, \delta)\}.
	\end{align*}
	By Corollary~\ref{coro: advanced geographic ineq decoupling map},
	$\decou_{n,T_1,T_2}(A_n(\delta))\subseteq B^{(1)}_{T_1}(\epsilon, \delta)\times B^{(2)}_{T_2}(\epsilon, \delta)$. Hence, by Proposition~\ref{prop: proba ineq decoupling map}, we have 
	\begin{align*}
		\P(A_n(\delta))
		\leq \cdecou \P^{\otimes 2}(\decou_{n,T_1,T_2}(A_n(\delta))) 
		\leq \cdecou \P^{\otimes 2}\big(B^{(1)}_{T_1}(\epsilon, \delta)\times B^{(2)}_{T_2}(\epsilon, \delta)\big),
	\end{align*}
	where $\cdecou:=(n+1)^{2r+1}(\tau_K+1)^rc_K^r$ satisfies
	\begin{equation*}
		\lim_{n\to \infty }\frac 1n\log \cdecou=0.
	\end{equation*}
	In this inequality, we take the logarithm of both sides, divide by $n$, and we take the limit superior as $n\to \infty$.
	We get
	\begin{align*}
		\overline s(\Bc_\LP(\mu, \delta) )
		&\leq \lim_{n\to \infty}\frac 1n\log \cdecou + \limsup_{n\to \infty}\frac 1n \log\P\big(B_{T_1}^{(1)}(\epsilon, \delta)\big) + \limsup_{n\to \infty}\frac 1n\log \P\big(B_{T_2}^{(2)}(\epsilon, \delta)\big)
		\\&=0+(\lambda_1+\epsilon)\overline s(\Bc_\LP(\mu_1,f_1(\epsilon,\delta)))+(\lambda_2+\epsilon)\overline s(\Bc_\LP(\mu_2,f_2(\epsilon,\delta))).
	\end{align*}
	Since $\IRL(\mu)$, $\IRL(\mu_1)$ and $\IRL (\mu_2)$ exist, by Proposition~\ref{prop: IRL for admissible measures},
	taking $\delta\to 0$ and then $\epsilon \to 0$ yields
	\begin{equation*}
		-\IRL(\mu)\leq -\lambda_1\IRL(\mu_1)-\lambda_2\IRL(\mu_2).
	\end{equation*}
	This concludes the proof.
\end{proof}
\appendix
\renewcommand{\thesubsection}{\Alph{subsection}}
\renewcommand{\thelem}{\thesubsection\arabic{lem}}
\setcounter{equation}{0}
\renewcommand{\theequation}{\thesubsection.\arabic{equation}}
\counterwithin{theo}{subsection}
\section*{Appendix}
\addcontentsline{toc}{section}{Appendix}
\subsection{A measure theoretic lemma}
We recall here the statement and the proof of Lemma~\ref{lem: measurable density}, which are standard \cite{bogachev,kallenberg}. 
\begin{lem}
	\label{lem: measurable density}
	Let $p$ be a stochastic kernel on $\R^d$ such that $p(x,\cdot)$ is absolutely continuous with respect to $\leb$ for all $x\in \R^d$. Then, $p$ has a jointly measurable density; in other words there exists a function $f:\R^d\times \R^d\to[0, \infty)$ that is $\Bc(\R^d\times \R^d)$-measurable\footnote{$\Bc(\R^d\times \R^d)$ denotes the Borel $\sigma$-algebra of $\R^d\times \R^d$.} and such that 
	\begin{equation*}	
		p(x,A)=\int_Af(x,y)\leb (\d y),\qquad x\in \R^d, A\in \Bc(\R^d).
	\end{equation*}
\end{lem}
\begin{proof}
	By the Radon-Nikodym Theorem, for all $x\in \R^d$, the measure $p(x,\cdot)$ has a density, which we denote $g(x,\cdot)$. The measurability of $g$ as a function of $\R^d \times \R^d$ may fail if functions $g(x,\cdot)$ are not precisely chosen.
	Let $\mathcal D_n$ denote a countable partition of $\R^d$ into cubes of side $2^{-n}$, 
	and set 
	\begin{equation*}
		f_n(x,y):=2^{dn}\sum_{C\in \mathcal D_n}\mathbf 1_C(y)p(x,C)
		,\qquad x,y\in \R^d.
	\end{equation*}
	Since it is a countable sum of measurable functions, $f_n$ is measurable. 
	Hence, the pointwise limit superior of $f_n$ is also measurable. We denote it $f$ and we  show that it is a density of $p$. Let $x\in \R^d$.
	For all $y\in\R^d $, let $C_n(y)$ denote the cube of $\mathcal D_n$ containing $y$. We have
	\begin{equation*}
		f_n(x,y)=2^{dn}p(x,C_n(y))=2^{dn}\int_{C_n(y)}g(x,z)\leb (\d z),\qquad y\in \R^d.
	\end{equation*}
	By the Lebesgue differentiation theorem, $(f_n(x,\cdot))$ converges almost everywhere to $g(x,\cdot)$, implying  that $f(x,y)=g(x,y)$ for almost every $y\in\R^d$ (while $x\in \R^d$ is fixed). Therefore, for all $A\in \Bc(\R^d)$,
	\begin{equation*}
		\int_{A}f(x,y)\leb(\d y)=\int_Ag(x,y)\leb (\d y)=p(x,A).
	\end{equation*}
\end{proof} 
\subsection{The L\'evy-Prokhorov metric}
\label{section: levy prokhorov}
The L\'evy-Prokhorov metric is defined on $\Pc(\R^d)$ by 
\begin{equation}
	\label{eq: levy prokhorov}
	\dlp(\mu, \nu)=\inf\{\delta>0\ |\ \forall A\subseteq \R^d \ \hbox{Borel set},\,\mu(A^\delta)+\delta\geq \nu(A)\}, \qquad \mu, \nu\in \Pc(\R^d),
\end{equation}
where $A^\delta$ denotes the $\delta$-neighborhood of $A$.
The L\'evy-Prokhorov metric metrizes the weak topology on $\Pc(\R^d)$. It satisfies $\dlp\leq \dtv$ where $\dtv$ is the total variation distance. We recall that $\Bc_\LP(\mu,\delta)$ denotes the open L\'evy-Prokhorov ball of radius $\delta$ centered at $\mu$.
\begin{lem}
	\label{lem: dlp small implies similar support}
	Let $\mu\in \Pc(\R^d)$ be absolutely continuous with respect to the Lebesgue measure. Let $O$ be an open set such that $\mu(O)>0$. Then there exist $\kappa>0$ and $\delta>0$ such that $\nu(O)\geq\kappa>0$ for all $\nu\in \Bc_\LP(\mu, \delta)$.
\end{lem}
\begin{proof}
	By inner regularity of $\mu$ (the measure $\mu$ is inner regular because it is absolutely continuous with respect to the Lebesgue measure), there exists a compact set $K\subseteq O$ such that $\mu(K)\geq \frac23\mu(O)$. Let $\delta=\min(\frac 12d(K,O^c),\frac 13\mu(O))$ and let $\nu\in \Pc(\R^d)$ be such that $\dlp(\nu,\mu)\leq \delta$. Thus, 
	\begin{equation*}
		\nu(O)\geq \nu\big(\{x\in \R^d\ |\ d(x,K)\leq \delta\}\big)\geq \mu(K)-\delta\geq \frac 13\mu(O)=:\kappa.
	\end{equation*}
\end{proof}
\begin{lem}
	\label{lemma: dlp with convex sums}
	Let $\delta>0$ and $N\in \N$.
	Let $(\mu_i),(\nu_i)\in \Pc(\R^d)^N$ be such that $\dlp(\mu_i,\nu_i)<\delta$ for all $i$. We set
	\begin{equation*}
		\mu=\sum_{i=1}^N\lambda _i\mu_i,\qquad \nu=\sum_{i=1}^N\gamma _i\nu _i,
	\end{equation*}
	where $(\lambda_i)$ and $ (\gamma_i)$ are sequences of nonnegative real numbers such that $\lambda_1+\ldots+\lambda_N=\gamma_1+\ldots+\gamma_N=1$. Then,
	\begin{equation}
		\label{eqloc: bound dlp with convex sums}
		\dlp(\mu, \nu)\leq \delta+\sum_{i=1}^N|\lambda_i-\gamma_i|.
	\end{equation}
\end{lem}
\begin{proof}
	Let $A$ be a Borel set, and set $ B=\{x\in \R^d\ |\ d(x,A)\leq \delta+\sum_{i=1}^N|\lambda_i-\gamma_i|\}$.
	Since $B$ contains the $\delta$-neighborhood of $A$, we have $\mu_i(B)\geq \nu_i(A)-\delta$ for all $i$. Thus,
	\begin{align*}
		\mu(B)\geq \sum_{i=1}^N\lambda_i\nu_i(A)-\delta\geq\nu(A)-\sum_{i=1}^N|\gamma_i-\lambda_i|-\delta.
	\end{align*}
	This holds true for any Borel set $A$.
\end{proof}
\begin{lem}
	\label{lem: prop dlp on classes}
Let $\epsilon,\delta,\mu=\lambda_1\mu_1+\lambda_2\mu_2, C^{(1)},C^{(2)},K^{(1)},K^{(2)}$ be as in Section~\ref{section: decoupling}. 
	Let $\nu\ll \leb$ be such that $\dlp(\mu,\nu)< \delta$. 
		For $\gamma\in\{1,2\}$, we let $\alpha_\gamma=\nu(C^{(\gamma)})>0$ and $\nu_\gamma=\frac1{\alpha_\gamma}\nu(\cdot\cap C^{(\gamma)})$.
	In addition, let $\alpha_3=1-\alpha_1-\alpha_2$ and $\nu_3$ denote either the normalized restriction of $\nu$ to $\R^d\setminus (C^{(1)}\cup C^{(2)})$ if $\alpha_3\neq 0$, or any probability measure on $\R^d\setminus (C^{(1)}\cup C^{(2)})$ otherwise, so that 
	\begin{equation*}
		\nu=\alpha_1\nu_1+\alpha_2\nu_2+\alpha_3\nu_3.
	\end{equation*}
	Then, for $\gamma\in \{1,2\}$,
	\begin{align}
		\label{eq: dlp prop second}
		\dlp(\mu_\gamma,\nu_\gamma)&\leq \frac{\epsilon}{\lambda_\gamma}(1+{\lambda_{3-\gamma}})+\delta\Big(1+\frac{1}{\lambda_\gamma}\Big).
	\end{align}
\end{lem}
\begin{proof}
	The fact that $\alpha_1$ and $\alpha_2$ are positive is a consequence of $\dlp(\mu, \nu)< \delta$ and the definition of $\delta$, from which we have
	\begin{equation}
		\label{eqloc: dlp prop first}
		\alpha_\gamma\geq \lambda_\gamma-\epsilon>0,\qquad \gamma\in \{1,2\}.
	\end{equation}
	We denote by $m$ the quantity on the right-hand side of~\eqref{eq: dlp prop second}. 
	Let $A$ be a Borel set of $\R^d$ and let $A'$ and $A''$ respectively denote the $\delta$-neighborhood of $A$ and the $m$-neighborhood of $A$. To prove~\eqref{eq: dlp prop second} for $\gamma=1$, we must bound $\mu_1(A'')$ from below. Since $\mu_1(A''\cap C^{(2)})\geq 0= \nu_1(A\cap C^{(2)})$, we can assume that $A\subseteq C^{(1)}$ without loss of generality. Since $\dlp(\mu,\nu)< \delta$, we have
	\begin{align}
		\mu_1(A'')\geq \mu_1(A')
		&=\frac1{\lambda_1}\mu(A')-\frac{\lambda_2}{\lambda_1}\mu_2(A')
		\\&\geq \frac1{\lambda_1}\nu(A)-\frac{1}{\lambda_1}\delta-\frac{\lambda_2}{\lambda_1}\mu_2(A')
		\nonumber
		\\&=\frac{\alpha_1}{\lambda_1}\nu_1(A)+\frac{\alpha_2}{\lambda_1}\nu_2(A)+\frac{\alpha_3}{\lambda_1}\nu_3(A)-\frac{1}{\lambda_1}\delta-\frac{\lambda_2}{\lambda_1}\mu_2(A').
		\label{eqloc: bound on mu1A''}
	\end{align}
	Let us estimate all five terms in~\eqref{eqloc: bound on mu1A''}.
	The first term can be bounded by using~\eqref{eqloc: dlp prop first}, which yields that ${\alpha_1}/{\lambda_1}\geq 1-{\epsilon }/{\lambda_1}$.
	The second and third terms are null because $\nu_2(C^{(1)})=\nu_3(C^{(1)})=0$.
	We leave the fourth term untouched.
	We have to bound the last term.
	On the one hand, $A'$ does not intersect $K^{(2)}$ because $d(A', K^{(2)})\geq d(C^{(1)},K^{(2)})-\delta>0$ by the assumption of $\delta$. On the other hand, 
	we have $\mu_2(K^{(2)})\geq 1-\epsilon$ by construction of $K^{(2)}$. This implies that $\mu_2(A')\leq \epsilon$ and bounds the last term. Altogether, these inequalities yield
	\begin{equation*}
		\mu_1(A'')\geq \nu_1(A)-\frac\epsilon{\lambda_1}-\frac{\delta}{\lambda_1}-\frac{\lambda_2}{\lambda_1}\epsilon=\nu_1(A)-m+\delta\geq \nu_1(A)-m.
	\end{equation*}
	The bound~\eqref{eq: dlp prop second} is proved for $\gamma=1$. Exchanging $1$ and $2$ yields~\eqref{eq: dlp prop second} for $\gamma=2$. 
\end{proof}
\subsection{Examples}
\label{section: examples}
In this section, we present three examples of non-irreducible Markov chains to which Theorems \ref{theo: main result} or \ref{theo: main result (simplified)} apply. Along the way, we present a technical result (Proposition \ref{prop: rorw}), which is a useful tool for proving~\ref{hyp: getting fast to classes}.

As mentioned in the introduction, non-irreducible Markov chains are commonly found in applications.  
Example~\ref{ex: lotka volterra} is one of the many examples of competition models where extinctions~---~which are definitive by nature~---~make the  Markov chain inherently non-irreducible. Example~\ref{ex: lotka volterra} also provides an instance of non-convex rate function.

Example \ref{ex: rorw} illustrates the somewhat extreme case where there are no communicating classes. We show that Theorem \ref{theo: main result} applies and yields the weak LDP with rate function $I\equiv \infty$.

Example~\ref{ex: one dimensional} involves a dynamical system on $\R$ perturbed by some bounded additive noise.
In this example, the communicating classes are explicitly identified.
\begin{example}[Competitive model with extinctions]
	\label{ex: lotka volterra}
	We consider a system of $d\geq 2$ interacting species governed by competitive Lotka-Volterra dynamics with nonnegative interaction matrix $(a_{i,j})$ and positive intrinsic growth rates $r_1,\ldots,r_d$. The population dynamics is described by the system of differential equations
	\begin{equation*}
		\frac{\d x_{i}}{\d t}(t)=r_ix_i(t)\bigg(1-\sum_{j=1}^da_{i,j}x_j(t)\bigg), \qquad 1 \leq i\leq d.
	\end{equation*}
	Let $F:\R\times \R^d_+\to \R^d_+$ denote the flow of the system. 
	We fix an initial measure $\beta\in\Pc(\R^d_+)$, which we assume is absolutely continuous with respect to the Lebesgue measure, with lower semicontinuous and bounded density.
	If nothing were to interfere with the dynamics, the distribution of the populations at time $t$ would be $\beta\circ F(t,\cdot)^{-1}\in \Pc(\R_+^d)$.
	However, we shall now introduce random perturbations.
	Let $(Z_{n,i})$ be a family of i.i.d.~positive, absolutely continuous, random variables and let $g$ denote their common density. We assume that at each integer time, a catastrophic event occurs (pesticide application, hunting, flood, etc), effectively reducing the population of species $i$ by an amount $Z_{n,i}$.
	Let $X_n = (X_{n,1}, \dots, X_{n,d}) \in \mathbb{R}^d$ denote the population size of each species immediately after the $n$-th perturbation. 
	If $X_{n,i}\leq 0$, we say that species $i$ is \emph{extinct} at time $n$.\footnote{Instead of representing extinct species by a null population size, we choose to represent them by arbitrary nonpositive values. We do this to maintain absolute continuity with respect to the Lebesgue measure without modifying the dynamics. In the interpretation, negative values should be understood as zero.}
	Extinction is definitive, and extinct species do not influence the dynamics anymore.
	The population dynamics is driven by the relation
	\begin{equation*}
		X_{n+1,i}=F_i(1,X_{n}^+)-Z_{n+1,i},\qquad 1\leq i\leq d,
	\end{equation*}
	where $x^+=(\max(x_1,0),\ldots, \max(x_d,0))$ for all $x\in\R^d$. 
	Notice that whenever $X_{n,i}\leq 0$, we have $F_i(1,X_n^+)=0$ so that $X_{n+1,i}=-Z_{n+1,i}\leq 0$ (extinctions are indeed definitive).
	The sequence $(X_n)_{n\geq 1}$ is a Markov chain. Since extinctions are definitive, the Markov chain is not irreducible.\footnote{Denoting by $\psi$ the common law of $-Z_{n,i}$, the Markov chain is actually $\psi^{\otimes d}$-irreducible. This remark allows us to use existing results of the literature to derive the large deviations of $(L_n)$ under the assumption that the initial measure (and hence the process at all time) is supported on $\R^d_-$ \cite{deacosta2022}. However, $\R^d_-$ is  precisely the domain where the Markov chain does not have a meaningful practical interpretation, as a point in $\R^d_-$ corresponds to all species being extinct. The large deviations with an arbitrary initial measure cannot be deduced without studying the theory of non-irreducible Markov chains}
	The density of the Markov kernel is given by 
	\begin{equation*}
		\rho(x,y)=\prod_{i=1}^d g(F_i(1,x^+)-y_i).
	\end{equation*}
	If $g$ is lower semicontinuous and bounded, then so is $\rho$.
	The Markov chain satisfies the assumptions of Theorem~\ref{theo: main result (simplified)}.
	Hence, $(L_n)$ satisfies a weak LDP with rate function $I$ satisfying the conclusions of Theorem~\ref{theo: main result}.
	Let us discuss admissibility.
	There are $2^d$ communicating classes, indexed by the subsets of $\{1,\ldots,d\}$ indicating which species are extinct, and the relation $(\leadsto)$ is the inclusion.
	Let $\mu\in \Pc(\R^d)$ be absolutely continuous with respect to the Lebesgue measure. Given $I\subseteq \{1,\ldots, d\}$, we set $A_{i,I}=\R_-$ if $i\in I$ and $A_{i,I}=\R_+$ if $i\notin I$. 
	The measure $\mu$ is admissible if and only if, for all $I_1,I_2\subseteq \{1,\ldots ,d\}$ such that $\mu(A_{1,I_1}\times\ldots \times A_{d,I_1})>0$ and $\mu(A_{1,I_2}\times\ldots \times A_{d,I_2})>0$, we have $I_1\subseteq I_2$ or $I_2\subseteq I_1$.
	
	We now consider the simple case $d=2$ and show that the rate function is not convex. Assume that $a_{1,1},a_{2,2}>0$ and that $g$ is positive on some interval $(0,\epsilon)$, and let us find two measures $\mu_1\in \Pc(\R_+\times \R_-)$ and $\mu_2\in \Pc(\R_-\times \R_+)$ such that $\IRL(\mu_1)$ and $\IRL(\mu_2)$ are both finite. For all $x\in K_1:=[\frac {r_1}{2a_{1,1}}, \frac {r_1}{a_{1,1}}]\times[-1,0] $, we have 
	\begin{equation*}
		p(x,K_1)\geq p\Big(\Big(\frac {r_1}{2a_{1,1}},0\Big), K_1\Big)=:\alpha_1>0.
	\end{equation*}
	Since $\Pc(K_1)$ is compact, the weak LDP upper bound yields
	\begin{equation*}
		-\inf_{\mu_1\in \Pc(K_1)}I(\mu_1)\geq \limsup_{n\to \infty}\frac 1n \log \P(L_n\in \Pc(K_1))\geq -\log \alpha_1>-\infty.
	\end{equation*}
	Therefore, there exists $\mu_1\in \Pc(\R_+\times \R_-)$ such that $I(\mu_1)<\infty$. By the same argument, there exists $\mu_2\in\Pc(\R_-\times \R_+)$ such that $I (\mu_2)<\infty$.
	Let $\mu=\frac12\mu_1+\frac12\mu_2$. The measure $\mu$ is not admissible, hence $I(\mu)=\infty$. 
	This shows that the rate function $I$ is not convex.
\end{example}

\begin{example}
	\label{ex: rorw}
	Let $d=1$. Let $0<\alpha<1$ and set $h:x\mapsto\mathbf 1_{(0,1)}(x)(1-\alpha)x^{-\alpha}$.
	Let $(E_n)$ be an i.i.d.~sequence of random variables of density $h$. 
	Let $\beta$ be a probability measure with bounded lower semicontinuous density with respect to the Lebesgue measure and let $(X_n)$ be the Markov chain defined by 
	\begin{equation*}
		X_1\sim \beta,\qquad 
		X_{n+1}=X_n+E_{n+1},\qquad n\geq 1.
	\end{equation*}
	In other words, the stochastic kernel is defined by the density 
	\begin{equation*}
		(x,y)\mapsto \rho(x,y):=h(y-x)=\mathbf1_{(x,x+1)}(y)\frac{1-\alpha}{(y-x)^\alpha}.
	\end{equation*} 
	Let us apply Theorem \ref{theo: main result}.
	Assumption \ref{hyp: pseudo lsc density} is satisfied because $\rho$ is lower semicontinuous on $(\R\cup \{\xinit\})\times \R$.
	The sequence $(X_n)$ is strictly increasing, thus there are no communicating classes and $C=\emptyset$, where $C$ is as in \eqref{eq: C}.
	Hence, Assumptions \ref{hyp: pseudo uniformity} and \ref{hyp: null border} are trivially satisfied.
	It remains to prove \ref{hyp: getting fast to classes}. In order to do so, we first consider the Markov chain $(Y_n)=(X_{kn})$, where $k=\lceil 1/(1-\alpha)\rceil$, which is easily shown to have bounded, lower-semicontinuous, kernel density.
	The Markov chain $(Y_n)$ is increasing.
	Let $U$ be an open interval and $\kappa>0$. 
	By Proposition~\ref{prop: rorw} below,
	\begin{equation}
		\label{eqloc: Mn(U) rorw}
		\lim_{n\to \infty}\frac 1n\log \P\Big(M_n(U)\geq \frac \kappa k\Big)=-\infty,
	\end{equation}
	where $M_n$ is the empirical measure of $(Y_n)$ up to time $n$.
	For $n\in \N$, since $(X_n)$ is increasing, if there are at least $\kappa n$ terms of $(X_1,\ldots, X_n)$ inside the interval $U$, then there are at least $\kappa n/k-1$ terms of $(Y_1,\ldots, Y_{\lfloor n/k\rfloor})$ inside $U$. Therefore,
	\begin{equation*}
		\P(L_n(U)\geq \kappa)\leq \P\Big(M_{\lfloor n /k\rfloor}(U)\geq \frac \kappa k-\frac kn\Big).
	\end{equation*}
	Therefore,
	\begin{equation}
		\label{eqloc: getting fast out of U sepcial exp case}
		\lim_{n\to \infty}\frac 1n\log \P(L_n(U)\geq \kappa)\leq \lim_{n\to \infty}\frac kn\log \P\Big(M_{n}(U)\geq \frac \kappa k\Big)=-\infty,
	\end{equation}
	further implying that Assumption~\ref{hyp: getting fast to classes} is satisfied.\footnote{For more details on why \eqref{eqloc: getting fast out of U sepcial exp case} implies \ref{hyp: getting fast to classes}, see the end of the proof of Proposition \ref{prop: rorw}.} By Theorem \ref{theo: main result}, $(L_n)$ satisfies the weak LDP with rate function $I\equiv \infty$.\footnote{Notice that, in this example, the weak LDP with rate function $I\equiv \infty$ can also be deduced directly from \eqref{eqloc: getting fast out of U sepcial exp case}.}
	\end{example}
	The following proposition is used in the example above. We will also use it later in Example~\ref{ex: one dimensional}. 
	\begin{prop}
		\label{prop: rorw}
		Let $d=1$.
		Assume that $(X_n)$ is increasing, almost surely.
	Assume, in addition, that $\rho$ is bounded on $(\R\cup \{\xinit\})\times \R$. 
	Then, for any open bounded interval $U\subseteq \R$ and any $\kappa>0$,  
	\begin{equation}
		\label{eq: getting out of U fast}
		\lim_{n\to \infty}\frac 1n\log \P(L_n(U)\geq \kappa)=-\infty.
	\end{equation} 
	In particular, Assumption~\ref{hyp: getting fast to classes} is satisfied and $(L_n)$ satisfies the weak LDP with rate function $I\equiv \infty$. The same conclusions hold if $(X_n)$ is strictly decreasing.
\end{prop}
\begin{proof}
	Let $\kappa>0$.
	We first assume that $U$ is of the form $U=(x_0,x_0+a)$ where $x_0\in \R$ and $a>0$ is such that $\rho(x,y)\leq 1/a$ for all $x,y\in \R$. 
	Let $T$ denote the hitting time of $U$. If $T$ is infinite, then $L_n(U)=0$ for all $n\in \N$, hence we assume $T<\infty$.
	We set 
	\begin{equation*}
		Z_0=X_T,\qquad Z_k=X_{T+k}-X_{T+k-1},\qquad k\geq 1.
	\end{equation*}
	Since $(X_k)$ is (strictly) increasing, the increments $Z_k$ are positive for $k\geq 1$.
	We have
	\begin{align*}
		\P(L_n(U)\geq \kappa)\leq \P(Z_1+\ldots+Z_{\lceil \kappa n\rceil}< a).
	\end{align*}
	Let $\gamma$ denote the law of $Z_0$ and let $\varphi_k(z)=\rho(z_0+\ldots +z_{k-1},z_0+\ldots +z_{k})$ for $z\in \R^{ \lceil \kappa n\rceil+1}$, so that
	\begin{align*}
		\P(Z_1+\ldots+Z_{\lceil \kappa n\rceil}< a)
		&=\int_{\R^{\lceil \kappa n\rceil+1}}\mathbf 1_{(0,a)}\Bigg(\sum_{i=1}^{\lceil \kappa n\rceil}z_i\Bigg)\gamma(\d z_0)
		\prod_{k=1}^{{\lceil \kappa n\rceil}}\mathbf 1_{(0,a)}(z_k)\varphi_{k}(z)\leb(\d z_k).
	\end{align*}
	In this expression, the factors $\mathbf 1_{(0,a)}(z_k)$ underline the fact that the event  $\{Z_1+\ldots+Z_{\lceil \kappa n\rceil}<a  \}$ is incompatible with any event $\{Z_k\geq a\}$ because the increments $Z_k$ are positive.
	Since $\varphi_k(z)\leq 1/a$ by definition, we have
	\begin{align*}
		\P(Z_1+\ldots+Z_{\lceil \kappa n\rceil}<a )
		&\leq \int_{\R^{\lceil \kappa n\rceil+1}}\mathbf 1_{(0,a)}\Bigg(\sum_{i=1}^{\lceil \kappa n\rceil}z_i\Bigg)\gamma(\d z_0)
		\prod_{k=1}^{{\lceil \kappa n\rceil}}\frac 1a\mathbf 1_{(0,a)}(z_k)\leb(\d z_k)
		\\&=\P(Z'_1+\ldots+Z'_{\lceil \kappa n\rceil}< a),
	\end{align*}
	where $(Z'_k)$ is a sequence of i.i.d.~random variables of uniform law over $(0,a)$.
	Let $\epsilon>0$. For $n$ large enough, we have $a\leq \epsilon \lfloor \kappa n\rfloor$. By Cram\'er's Theorem \cite[Theorem 2.2.3]{DZ},
	\begin{equation}
		\begin{split}
		\label{eqloc: cramers theorem}
		\limsup_{n\to \infty}\frac 1n\log\P(Z'_1+&\ldots+Z'_{\lceil \kappa n\rceil}< a)
		\\&\leq
		\limsup_{n\to \infty}\frac 1n\log\P\Big(\frac 1{\lceil \kappa n\rceil}(Z'_1+\ldots+Z'_{\lceil \kappa n\rceil})< \epsilon\Big)
		=-\kappa \Lambda^*(\epsilon),
		\end{split}
	\end{equation}
	where 
	\begin{equation*}
		\Lambda^*(\epsilon)=\sup_{\lambda\in \R}\Bigg(\lambda \epsilon-\log\int_0^a \frac1ae^{\lambda x}\d x\Bigg)=\sup_{\lambda\in \R^*}\Big(\lambda \epsilon-\log \frac{e^{\lambda a}-1}{\lambda a}\Big).
	\end{equation*}
	Since $\Lambda^*$ is lower semicontinuous and $\Lambda^*(0)=\infty$, taking the limit as $\epsilon\to 0$ in~\eqref{eqloc: cramers theorem} yields
	\begin{equation*}
		\limsup_{n\to \infty}\frac 1n\log\P(Z'_1+\ldots+Z'_{\lceil \kappa n\rceil}< a)=-\infty.
	\end{equation*}
	This completes the proof of~\eqref{eq: getting out of U fast} for any $\kappa>0$ and any interval of the form $U=(x_0,x_0+a)$ where $x_0\in \R$ and $a>0$ is such that $\rho(x,y)\leq 1/a$ for all $x,y\in \R$.
	In the general case, writing $U$ as the union of overlapping such small intervals $U_1,\ldots, U_s$ and using~\eqref{eq: getting out of U fast} for each $U_i$ and $\kappa'=\kappa/s$ yields
	\begin{align*}
		\limsup_{n\to \infty}\frac 1n\log \P(L_n(U)\geq \kappa)
		&\leq \limsup_{n\to \infty}\frac 1n \log\sum_{i=1}^s\P\Big(L_n(U_i)\geq \frac \kappa s\Big)
		\\&=\max_{1\leq i\leq s}\limsup_{n\to \infty}\frac 1n\log \P\Big(L_n(U_i)\geq \frac \kappa s\Big)=-\infty.
	\end{align*}
	Therefore,~\eqref{eq: getting out of U fast} holds for any open bounded interval $U$.
	Let us now prove~\ref{hyp: getting fast to classes}. By the strict monotonicity of $(X_n)$, we have $C=\emptyset$ and $\mu(C)=0$ for all $\mu \in \Pc(\R)$. Proving~\ref{hyp: getting fast to classes} is proving Equation~\eqref{eq: getting fast to classes} for all $\mu\in \Pc(\R)$.
	Let $\mu\in \Pc(\R)$ and assume that $\mu\ll \leb$ (otherwise, Proposition~\ref{prop: nonadmissible: abscont} yields the conclusion).
	There exists an open bounded interval $U$ such that $\mu(U)>0$. By Lemma~\ref{lem: dlp small implies similar support}, there exist $\kappa>0$ and $\delta>0$ such that $\nu(U)\geq \kappa $ for all $\nu \in \Bc_\LP(\mu, \delta)$.
	Therefore, by~\eqref{eq: getting out of U fast},
		\begin{equation*}
		\limsup_{n\to \infty}\frac 1n\log \P(L_n\in \Bc_\LP (\mu, \delta))\leq \limsup_{n\to \infty}\frac 1n\log \P(L_n(U)\geq \kappa)=-\infty.
	\end{equation*}
	Equation~\eqref{eq: getting fast to classes} holds with $V=\Bc_\LP(\mu, \delta)$. As a consequence, the weak LDP holds with rate function $I=\IRL(\mu)\equiv \infty$.
\end{proof}
\begin{example}
	\label{ex: one dimensional}
	Let $d=1$ and $f:\R\to \R$ be a continuous nondecreasing function. Let $\varphi$ be a lower semicontinuous bounded probability density such that $\varphi^{-1}((0,\infty))=(-1,1)$.
	We consider a Markov chain whose stochastic kernel is given by
	\begin{equation*}
		\rho(x,y)=\varphi(y-f(x)),\qquad x,y\in \R.
	\end{equation*}
	In other words, for all $n\in \N$,
	\begin{equation*}
		X_{n+1}=f(X_n)+Z_{n+1},
	\end{equation*}
	where $(Z_i)$ is an i.i.d.~sequence of random variables of law $\varphi$ independent of $X_1$. 
	See Figure~\ref{fig: example dynamic system} for a visual representation of the dynamics. 
	\begin{figure}[!tb]
		\centering 
		\includegraphics[width=0.6\textwidth]{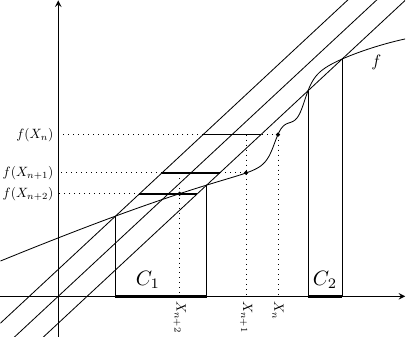}
		\captionsetup{width=.9\linewidth}
		\caption[width=0.8\textwidth]{The Markov chain of Example~\ref{ex: one dimensional} for an example function $f$. The three parallel lines are the lines of equations $y=x-1$, $y=x$ and $y=x+1$.
			The thick horizontal segment at $y=f(X_n)$ (resp. $y=f(X_{n+1})$ and $y=f(X_{n+2})$) denotes the support of $X_{n+1}$ (resp. $X_{n+2}$ and $X_{n+3}$).
			There are two classes $C_1$ and $C_2$, and $2\leadsto 1$. To the left of $C_1$, the Markov chain can only move to the right. Between $C_1$ and $C_2$ and to the right of $C_2$, the Markov chain can only move to the left.}
	\label{fig: example dynamic system}
	\end{figure}
	Let us study the behavior of such Markov chain. We shall begin the analysis by the observation that points $a\in \R$ such that $f(a)-a\leq -1$ or $f(a)-a\geq 1$ work like `ratchets' in the system.
	\begin{lem}
		\label{lem: ratchet prop}
		Let $a\in \R$. If $f(a)-a\leq -1$, then all $x\leq a$ satisfy $\tilde p(x,[a,\infty))=0$.
		 If $f(a)-a\geq 1$, then all $x\geq a$ satisfy $\tilde p(x,(-\infty,a])=0$.
	\end{lem}
	\begin{proof}
		If $a\in \R$ is such that $f(a)-a\leq -1$, since $f$ is nondecreasing, all $x\leq a$ satisfy $f(x)+1\leq a$, implying that $p(x, [a,\infty))=0$. Thus, by induction, if $X_n\leq a$ for some $n$, then $X_k< a$ for all $k> n$. In other words,
		all $x\leq a$ satisfy $\tilde p(x,[a,\infty))=0$. The proof for the case $f(a)-a\geq 1$ is symmetrical.
	\end{proof}
	A first consequence of this observation is that $C\subseteq B:=\{x\in \R\ |\ f(x)-x\in(-1,1)\}$, where $C$ is as in~\eqref{eq: C}. Indeed, all $x\in \R\setminus B$ satisfy either $\tilde p(x,(-\infty,x])=0$ or $\tilde p(x, [x,\infty))=0$. In both cases, $\tilde \rho (x,x)=0$ thus $x\notin C$. 
	We now describe the class structure of this Markov chain in Proposition~\ref{prop: 1d example communicating classes}, and use Theorem~\ref{theo: main result (simplified)} to prove that $(L_n)$ satisfies the weak LDP in Proposition~\ref{prop: 1d example thm}.
	\begin{prop}
	\label{prop: 1d example communicating classes}
	The communicating classes are the connected components of $B$. 
	Two consecutive communicating classes $U_1,U_2$ with $\sup U_1\leq \inf U_2$ satisfy $U_1\leadsto U_2$ (resp. $U_2\leadsto U_1$) if and only if $f(x)-x\geq 1$ for all $x\in[\sup U_1,\inf U_2]$ (resp. $f(x)-x\leq -1$ for all $x\in[\sup U_1,\inf U_2]$).
	\end{prop}

\begin{proof}
	We begin by proving that, if $U$ is a connected component of $B$, then $\tilde \rho>0$ on $U^2$. Let $x\in U$ and let $A_x$ denote the open set $\{y\in U\ |\ \tilde \rho(x,y)>0\}$. 
	Since $\rho(x,x)=\varphi(x-f(x))>0$, the set $A_x$ contains $x$. 
	We prove that $A_x$ is closed. Let $z\in U\setminus A_x$. We have $\tilde \rho(z,z)>0$, thus, by Lemma~\ref{lem: positive density}, $\tilde \rho(\cdot,z)>0$ on a neighborhood $V\subseteq U$ of $z$. If some $y\in V$ satisfies $y\in A_x$, then both $\tilde \rho(x,y)>0$ and $\tilde \rho(y,z)>0$, thus $\tilde \rho(x,z)>0$ by the transitivity property of Lemma~\ref{lem: positive density}. This is a contradiction with $z\notin A_x$, hence we proved that $V\subseteq U\setminus A_x$. Therefore, $A_x$ is closed.
	Since $U$ is connected and $A_x$ is not empty, we have $A_x=U$, further implying that $\tilde \rho$ is positive on $U^2$.
	This proves that $C=B$. To complete the identification of communicating classes as the connected components of $B$, it only remains to prove that two such connected components are not parts of the same class.
	
	Let $U_1, U_2$ be two distinct connected components of $B$. We use Lemma~\ref{lem: ratchet prop} to prove that $U_1$ and $U_2$ are not part of the same communicating class by showing that $U_1\not \leadsto U_2$ or $U_2\not \leadsto U_1$. Without loss of generality, we assume that $a:=\sup U_1\leq \inf U_2<\infty$. Then, $f(a)-a\in\{-1,1\}$. 
	If $f(a)-a=-1$, it is impossible to reach $U_2$ from any point of $U_1$ because $U_2\subseteq [a,\infty)$. 
	If $f(a)-a=1$, it is impossible to reach $U_1$ from any point $x\geq a$ because $U_1\subseteq (-\infty,a]$; in particular, it is impossible to reach $U_1$ from any point of $U_2$. 
	This shows that the communicating classes are exactly the connected components of $B$.
	
	We now describe the partial order on the set of communicating classes. Let $U_1$ and $U_2$ be two consecutive connected components of $B$, such that $a:=\sup U_1\leq \inf U_2=:b$ and $f(x)-x\geq1$ for all $x\in [a,b]$.
	By continuity of $f$, we have $f(a)=a+1$ and $f(b)=b+1$. 
	Let $x_0\in U_1$ and let us show that $\tilde p(x_0,U_2)>0$. Since $U_1$ is a communicating class, we can choose $x_0$ arbitrarily close to $a$ without loss of generality. In particular, we can assume that $x_0$ is such that $f(x_0)>a$, since $f$ is continuous and $f(a)=a+1$.
	Let $(x_k)$ be the real sequence defined by $x_{k+1}=f(x_k)$. We have $\tilde \rho(x_0,x_k)>0$ and $x_{k+1}-x_k\geq1$ for all $k$ such that $x_k\leq b$. Let $n$ be the first index such that $f(x_{n})> b$. Since $f$ is nondecreasing, we have $b< f(x_{n})\leq f(b)=b+1$. In particular, we have $\rho(x_{n},y)>0$ for all $y\in (b,b+1)$. The set $(b,b+1)$ intersects $U_2$, thus $\tilde p(x_0,U_2)>0$. By a similar reasoning, we prove that $U_2\leadsto U_1$ if $f(x)-x\leq-1$ for all $x\in [a,b]$.
	\end{proof}
\begin{prop}
	\label{prop: 1d example thm}
	The Markov chain $(X_n)$ satisfies the assumptions of Theorem~\ref{theo: main result (simplified)}.
	The sequence $(L_n)$ satisfies the weak LDP with rate function $I$ satisfying the conclusions of Theorem~\ref{theo: main result}.
\end{prop}
	\begin{proof}
	The function $\rho$ is lower semicontinuous and bounded on $(\R\cup\{ \xinit\})\times \R$ by definition of $\varphi$.
	Since $C$ is the countable union of open intervals, $\partial C$ is countable thus~\ref{hyp: null border} is satisfied. 
	It remains to prove~\ref{hyp: getting fast to classes}.
	
	Let $\mu\in \Pc(\R)$ be absolutely continuous with respect to the Lebesgue measure and be such that there exists an open bounded interval $U\subseteq \R\setminus \overline C$ satisfying $\mu(U)>0$. We seek a neighborhood of $\mu$ satisfying~\eqref{eq: getting fast to classes}. 	
	Since $f$ is continuous, either $U\subseteq \{x\in \R\ |\ f(x)-x\geq1\}$ or $U\subseteq \{x\in \R\ |\ f(x)-x\leq-1\}$.
	Since the second case mirrors the first case, we only discuss the first case.
 	By Lemma~\ref{lem: dlp small implies similar support}, there exist $\kappa>0$ and $\delta>0$ such that $\{L_n\in \Bc_\LP (\mu, \delta)\}\subseteq \{L_n(U)\geq\kappa\}$. Let us prove that $\Bc_\LP(\mu, \delta)$ satisfies~\eqref{eq: getting fast to classes}.
		Let $T=\inf \{n\in \N\ |\ X_n \in U\}$.
		Since $\{L_n\in \Bc_\LP (\mu, \delta)\}\subseteq \{T<\infty\}$, we can assume that $T<\infty$.
		Let $(Y_n)$ be the Markov chain defined by
		\begin{equation*}
			Y_1=X_T,\qquad
			Y_{k+1}=
			\begin{cases}
				X_{T+k},\quad &\hbox{if $Y_k\in U$};
				\\Y_k+\xi_{k+1},\quad &\hbox{otherwise},
			\end{cases}
		\end{equation*}
		where $(\xi_{k})$ is an i.i.d.~sequence of uniform law on $(0,1)$ independent of $(X_n)$. 
		The Markov chain $(Y_n)$ satisfies the hypothesis of Proposition~\ref{prop: rorw}, thus 
		\begin{equation}
			\label{eqloc: Mk(U)}
			\lim_{k\to \infty}\frac 1k\log \P(M_k(U)\geq \kappa)=-\infty,
		\end{equation}
		where $M_k$ denotes the empirical measure of $(Y_k)$ up to time $k$. We now compare $L_n(U)$ and $M_n(U)$.
		Let $T'=\inf \{n\geq T\ |\ X_n \notin U\}$. We have $\{T'=\infty\}\subseteq \{\forall k\in \N,\ M_k(U)=1\}$ thus $T'$ is almost surely finite.
		By Lemma~\ref{lem: ratchet prop},
		we have $\tilde p(X_{T'},U)=0$,
		 implying that the times $n\in\N $ such that $X_n\in U$ are exactly the times $n\in \{T,\ldots, T'-1\}$.
		 Therefore, we always have $L_n(U)\leq M_n(U)$.
		Hence,
		\begin{align*}
			\P(L_{n}\in \Bc_\LP(\mu, \delta))
			&\leq \P(L_{n}(U)\geq \kappa)
			\leq \P(M_{n}(U)\geq \kappa).
		\end{align*}
		By~\eqref{eqloc: Mk(U)}, the neighborhood $\Bc_\LP(\mu, \delta)$ satisfies~\eqref{eq: getting fast to classes}.
	Assumption~\ref{hyp: getting fast to classes} is satisfied.
	By Theorem~\ref{theo: main result (simplified)}, $(L_n)$ satisfies the weak LDP with rate function satisfying the conclusions of Theorem~\ref{theo: main result}.
\end{proof}
\end{example}
In the special case $f=\mathrm{id}$, Example~\ref{ex: one dimensional} describes a random walk on $\R$ with bounded increments. Proposition~\ref{prop: 1d example communicating classes} shows that there is only one communicating class, which is the whole state space. Notice that neither (H*) of \cite{DV3} nor~\ref{assumption U} is satisfied. However, Theorem~\ref{theo: main result} applies.
\section*{Acknowledgments}
The author would like to express his gratitude to Noé Cuneo and Armen Shirikyan for valuable discussions and advice.
The author would also like to thank Pierre Petit for a helpful suggestion on an example.
\bibliographystyle{unsrt}
\bibliography{bf.bib}
\end{document}